\documentclass[a4paper]{article}
\usepackage{amsthm}
\usepackage{amssymb}
\usepackage{amsmath}
\usepackage{verbatim}
\usepackage{enumerate,url}
\usepackage[numbers]{natbib}
\usepackage{graphicx,multicol,wrapfig}
\usepackage{tikz}

\renewcommand{\thefootnote}{\fnsymbol{footnote}}

\newtheorem{theorem}{Theorem}[section]
\newtheorem{lemma}[theorem]{Lemma}

\newtheorem{proposition}[theorem]{Proposition}
\newtheorem{remark}[theorem]{Remark}

\newtheorem*{example*}{Example}
\newtheorem*{theorem*}{Theorem}

\newtheorem*{remark*}{Remark}

\newtheorem{corollary}[theorem]{Corollary}
\newtheorem*{corollary*}{Corollary}

\newtheorem{definition}[theorem]{Definition}

\newtheorem*{definition*}{Definition}

\newtheorem*{notation*}{Notation}

\newtheorem{notation}[theorem]{Notation}

\numberwithin{equation}{section}

\def\ba{\mathbf{a}}

\newcommand{\CC}{{\mathbb C}}

\newcommand{\RR}{{\mathbb R}}

\newcommand{\PP}{{\mathbb P}}

\newcommand{\NN}{{\mathbb N}}

\newcommand{\calA}{\mathcal{A}}

\newcommand{\calR}{\mathcal{R}}

\newcommand{\calS}{\mathcal{S}}

\renewcommand{\Im}{\mbox{Im}}

\renewcommand{\Re}{\mbox{Re}}

\newcommand{\co}{\mbox{co}}
\newcommand{\bc}{\mathbf{c}}
\newcommand{\bz}{\mathbf{z}}

\def \bar{\overline}
\def \hat{\widehat}

\def \b0{{\bf 0}}
\def\be{{\bf e}}

\def \span{\mathrm{span}}

\def\calL{\mathcal{L}}
\def\calE{\mathcal{E}}
\def\intt{\mathrm{int}}

\long\def\symbolfootnote[#1]#2{\begingroup%
\def\thefootnote{\fnsymbol{footnote}}\footnote[#1]{#2}\endgroup}

\begin{document}

\title{Extremal functions for real convex bodies: simplices, strips, and ellipses}
\author{Sione Ma`u}

\maketitle

\begin{abstract}
We present an explicit method to compute the (Siciak-Zaharjuta) extremal function of a real convex polytope in terms of supporting simplices and strips.  We use this to give a new proof of the existence of extremal ellipses associated to the extremal function of a real convex body.
\end{abstract}


\section{Introduction}

Let $K\subset\CC^d$ be compact.  The \emph{(Siciak-Zaharjuta) extremal function} of $K$ is defined by 
\begin{equation} \label{eqn:szextremal}
V_K(z):=\sup\left\{ \frac{1}{\deg p}\log^+|p(z)|\colon p\in\CC[z],\ \|p\|_K\leq 1 \right\};
\end{equation}
here $\|p\|_K=\sup_{z\in K}|p(z)|$ is the sup norm,  $\CC[z]$ denotes the (complex) multivariate polynomials in $z=(z_1,\ldots,z_d)$, and $\log^+|\cdot|=\max\{\log|\cdot|,0\}$.     The upper semicontinuous regularization $V_K^*(z)=\limsup_{\zeta\to z} V_K(\zeta)$ is either identically $+\infty$ or is a plurisubharmonic (psh) function of logarithmic growth that is maximal outside the set $K$:
\[ \begin{aligned}
V_K^*(z) &\leq  \log^+|z| +C \ \hbox{ for some constant } C, \\
(dd^cV^*_K)^d &= 0 \ \hbox{ on } \CC^d\setminus K.
\end{aligned} \]
Here $(dd^c\cdot)^d$ denotes the \emph{complex Monge-Amp\`ere operator}, which, applied to a function $u$ of class $C^2$, is given by the formula
$$(dd^cu)^d = 4^d d!\det
\begin{bmatrix}     
\frac{\partial^2}{\partial z_j\partial\bar z_k}
\end{bmatrix}   dV$$
where $dV$ denotes $2d$-dimensional Euclidean volume in $\CC^d$. The extension of $(dd^c\cdot)^d$ to locally bounded psh functions gives a positive measure \cite{bedfordtaylor:dirichlet}.  Both $V_K^*$ and $(dd^cV^*_K)^d$, the \emph{complex equilibrium  measure of $K$},  are of fundamental importance in pluripotential theory and polynomial approximation.   For most sets of interest,  $K$ is \emph{regular}, i.e., $V_K^*=V_K$; this is true for the sets considered in this paper (cf. Section \ref{sec:pluripotential}), so we will usually disregard  the `$*$' superscript.

Although explicit computation of $V_K$ and $(dd^cV_K)^d$ is virtually impossible in general, progress has been made in understanding certain cases.   Much is known when $K\subset\RR^d\subset\CC^d$ is a real convex body (i.e., a convex set with nonempty interior in $\RR^d$).   Lundin \cite{lundin:extremal} studied the structure of $V_K$ when $K$ is a convex body that is \emph{symmetric (with respect to the origin)}, i.e., that satisfies $K=-K$.   He computed explicit formulas for $V_K$ and $(dd^cV_K)^d$ when $K\subset\RR^d$ is the unit ball.  Baran \cite{baran:plurisubharmonic} extended this study to certain nonsymmetric convex bodies obtained as images of symmetric ones by a quadratic mapping; in particular, he derived an explicit formula for the extremal function of the standard simplex (cf. equation (\ref{eqn:vs})).  

Lundin's method for computing $V_K$ when $K$ is symmetric involved the construction of complex ellipses on which $V_K$ is harmonic. 
 Burns, Levenberg and Ma`u \cite{burnslevmau:pluripotential} verified the existence of such ellipses for all (symmetric and non-symmetric) real convex bodies.  They used deep results of Lempert relating the extremal function of the closure of a bounded, strictly linearly convex domain to Kobayashi geodesics in an associated (dual) strictly linearly convex domain   (\cite{lempert:metrique}, \cite{lempert:intrinsic}, \cite{lempert:symmetries}); the ellipses were obtained as limits of Kobayashi geodesics using a convergence argument.  

In this paper, we show that pluripotential theory on real convex bodies may be developed in a self-contained way  without reference to Lempert theory.   
The main theorem of the paper is the following. 

\begin{theorem*}[Theorems \ref{thm:main} \&\ \ref{thm:mainN} (special cases) and Theorem \ref{thm:maingen}]
Let $K\subset\RR^d$ be a compact convex polytope. Then we have the formula
$$
V_K(z) = \max\{ V_{S}(z)\colon S \in\calS(K) \}.
$$
\end{theorem*}

Here $\calS(K)$ denotes a finite collection of sets whose common intersection is $K$.  Each $S\in\calS(K)$ is a simplex or a \emph{strip} (see Section \ref{sec:strips}).  The explicit construction of these simplices and strips takes up Sections \ref{sec:hyperplanes}, \ref{sec:simplices}, and \ref{sec:strips}.  

The methods we use to prove our main theorem also provide a new proof of the main result in \cite{burnslevmau:pluripotential} as a byproduct.  These methods are developed in the remaining  sections.  

 In Section \ref{sec:pluripotential} we review some basic results in pluripotential theory for later use. 

 In Section \ref{sec:explicit} we recall the explicit formulas for the extremal functions of the unit ball and the standard  simplex in $\RR^d\subset\CC^d$.  We also recall the Robin function $\rho_K$ and Robin indicatrix $K_{\rho}$ of a compact set $K$, and compute them explicitly for the real ball and simplex.   


In Section \ref{sec:hooke} we introduce \emph{Hooke ellipses} for a symmetric convex body $K$.  A Hooke ellipse is one that is centered at the origin and inscribed in the convex body.  By an explicit calculation, we verify that for the real unit ball $B_{\RR}\subset\RR^d$, the extremal function $V_{B_{\RR}}$ is harmonic on the complexification $E_{\CC}$ of any Hooke ellipse $E$ for $B_{\RR}$.  Precisely, there is a parametrization $f\colon\CC\setminus\{0\}\to E_{\CC}$ such that $V_{B_{\RR}}(f(\zeta))=\left|\log|\zeta|\right|$.

 By making a linear change of coordinates, this harmonicity result extends to Hooke ellipses for a solid ellipsoid.      We then prove the following `\emph{identity principle}' for Hooke ellipses:

\begin{theorem*}[Theorem \ref{prop:19}]
Let $K_1,K_2$ be compact sets, and let $E_1\subset K_1,E_2\subset K_2$ be Hooke  ellipses whose complexifications $E_{1,\CC},E_{2,\CC}$ are parametrized by maps $f_1,f_2$.  Suppose
$$
V_{K_1}(f_1(\zeta))=\log|\zeta|,\ V_{K_2}(f_2(\zeta))=\log|\zeta| \quad\hbox{for all }\zeta\in\CC\setminus\Delta.
$$
Further, suppose $z_0\in E_{1,\CC}\cap E_{2,\CC}$ and $V_{K_1}(z_0)=V_{K_2}(z_0)>0$.  Then $E_{1,\CC}=E_{2,\CC}$.  
\end{theorem*}


Taking $K_1=K_2=B_{\RR}$, complexifications of Hooke ellipses for $B_{\RR}$ give a foliation of $\CC^d\setminus B_{\RR}$ such that $V_{B_{\RR}}$ is harmonic on each leaf of the foliation. The foliation parameter is given by the quotient of the boundary of the Robin indicatrix $\partial B_{\RR,\rho}$ with respect to a circle action.  All information about the foliation can be packaged into a \emph{generalized Joukowski map} that maps $\CC^d\setminus B_{\RR,\rho}$ to $\CC^d\setminus B_{\RR}$.  Up to a normalization factor, this is a direct generalization of the classical Joukowski function that maps the exterior of the unit disk to the exterior of  $[-1,1]$ in the complex plane.  A similar foliation result holds for a solid ellipsoid by making a  linear change of coordinates. 



 In Section \ref{sec:newton} we use the square map to transform the foliation of $\CC^d\setminus B_{\RR}$ by complexifications of Hooke ellipses to a foliation of $\CC^d\setminus \Sigma_d$ by complexifications of so-called \emph{Newton ellipses} (here $\Sigma_{d}$ is the standard simplex in $\RR^d$), such that the restriction of $V_{\Sigma_d}$ to the complexification of a Newton ellipse is harmonic outside of $\Sigma_d$.   Composing with an appropriate linear map, the foliation of $\CC^d\setminus\Sigma_d$ transforms to a foliation of $\CC^d\setminus S$, for any $d$-dimensional simplex $S$. 

In Section \ref{sec:inscribed} we consider ellipses inscribed in a real compact convex polytope $K\subset\RR^d$ and relate them to Newton ellipses for supporting simplices.  Using the square map and Proposition \ref{thm:ellipsoid} (which follows from Theorem \ref{prop:19}), we prove Theorem \ref{thm:main} which is a special case of our main theorem.

  In Section \ref{sec:robinexpmap} we consider convex polytopes $K$ whose ($d-1$)-dimensional faces satisfy a linear independence condition (cf. Theorem \ref{thm:simplices}).  If our main theorem holds, we show that there exists  a foliation of $\CC^d\setminus K$ by ellipses on which $V_K$ is harmonic.  The foliation is constructed by selecting appropriate ellipses associated to the various simplices in $\calS(K)$.  Information about the foliation may be packaged into a so-called \emph{Robin exponential map}.\footnote{This terminology was  introduced in \cite{burnslevmau:exterior}.}

In Section \ref{sec:computing}, we extend our main theorem to cover all convex polytopes whose ($d-1$)-dimensional faces satisfy the above-mentioned linear independence condition (Theorem \ref{thm:mainN}).  The proof uses an inductive argument on the number of ($d-1$)-dimensional faces together with the results of Section \ref{sec:robinexpmap}.   We then present an algorithm to compute the extremal function of a compact convex polytope explicitly, using our main theorem together with the barycentric coordinate formula for the extremal function of a real simplex \cite{bosmauwaldron:extremal}.  We illustrate the algorithm on a couple of convex polygons in $\RR^2$.

In Section \ref{sec:blm} we use Theorem \ref{thm:mainN} and an approximation argument to give a new proof of the main result of \cite{burnslevmau:pluripotential}:

\begin{theorem*}[Theorem \ref{thm:blm}] 
Let $K\subset\RR^d$ be a compact convex body.  For each $z\in\CC^d\setminus K$ there exists a complexified ellipse   $E_{\CC}\subset\CC^d$  with parametrization  
\begin{equation*}
\CC^*\ni\zeta \stackrel{f}{\longmapsto} \ba + \bc\zeta+\bar\bc/\zeta\in\CC^d  \quad (\ba\in\RR^d,\ \bc\in\CC^d)
\end{equation*}
such that  
\begin{itemize}
\item[(i)]
$z=f(\zeta_z)$ for some $|\zeta_z|>1$,  \item[(ii)] $V_K(f(\zeta))=\log|\zeta|$ for all $|\zeta|>1$, and  
\item[(iii)] $E:=\{f(e^{i\theta})\colon \theta\in\RR\}$ is a real ellipse (or line segment if $\bc$ is real) that is inscribed in $K$.  \end{itemize}
\end{theorem*}

The above result for convex bodies extends naturally from the results of Lundin and Baran and the existence of Hooke and Newton ellipses.  Figure 1 outlines the argument.  As a consequence, we show that our main theorem applies to all compact convex polytopes (Theorem \ref{thm:maingen}).  We also use Theorem \ref{thm:blm} to prove another `max' formula  for the extremal function of a finite intersection of symmetric convex bodies (Theorem \ref{thm:symmetric}).

Using Theorem \ref{thm:maingen} together with results on the transformation of the extremal function under polynomial mappings \cite{levperera:global}, one can even compute a number of other extremal functions.  This method was used, together with a specialized version of the theorem in $\RR^2$, to compute the extremal function of a real torus in $\RR^3$ as a subset of its complexification in $\CC^3$ \cite{piazzon:extremal}.  

\begin{figure}
\begin{tikzpicture}
\path (0,0) node[align=center](x) {\small real\\ \small ball}
(3,0) node[align=center](y) {\small Hooke\\ \small ellipses}
(0,-2) node[align=center](s) {\small real\\ \small simplex}   
(3,-2) node[align=center](n){\small Newton\\ \small ellipses}
(0,-4) node[align=center](p) {\small special\\ \small polytope$^{(4)}$} 
(3,-4) node[align=center](r){\small Robin\\ \small ellipses$^{(5)}$}
(0,-6) node[align=center](c){\small convex\\ \small body $K$}
(3,-6) node[align=center](e){\small ellipses\\ \small where $V_K$\\ \small is harmonic};
\draw[->] (x) -- node[pos=0.5,above]{\footnotesize(1)}  (y);
\draw[<->](x) -- node[pos=0.5,right]{\footnotesize(2)}(s);
\draw[<->](y)-- node[pos=0.5,right]{\footnotesize(2)}(n);
\draw[->,dashed] (s)--(n);
\draw[->,dashed] (s) -- node[pos=0.5,right]{\footnotesize(3)}  (p);
\draw[->,dashed] (p) -- (r);
\draw[->,dashed] (n)--(r);
\draw[->] (p)-- node[pos=0.5,right]{\footnotesize(6)} (c);
\draw[->] (r)-- node[pos=0.5,right]{\footnotesize(7)}(e);
\draw[->,dashed] (c)--node[pos=0.5,above]{\footnotesize(8)}   (e);
\draw (8,-3) node[align=left] {\small\textsc{Notes}\\ \footnotesize(1) Proposition \ref{prop:16}.\\ \footnotesize (2) Quadratic map (cf. Definition \ref{def:91}).\\
\footnotesize (3) Theorems \ref{thm:main} and \ref{thm:mainN}.\\
\footnotesize (4) Faces are as in Theorem \ref{thm:simplices}.\\
\footnotesize (5) Given by (\ref{eqn:rexp}) (the Robin exponential map).\\
\footnotesize (6) Approximation of the extremal function\\
\footnotesize \quad\  (Proposition \ref{prop:61}, part 3).\\
\footnotesize (7) Convergence of holomorphic maps.\\
\footnotesize (8) Proof of Theorem \ref{thm:blm}. };
\end{tikzpicture}
\caption{A solid line indicates a direct link, while a dashed line uses previously established links.}
\end{figure}

--------------------------------------------------




\section{Hyperplanes, normal vectors, affine functions}\label{sec:hyperplanes}

In this section and the next, and through most of Section \ref{sec:strips}, everything takes place in $\RR^d$.  We  
 use the following standard notation: 
 $\span\{v_1,\ldots,v_m\}$ denotes the span over $\RR$ of the vectors $v_1,\ldots,v_m$ in $\RR^d$,  and 
 `${\perp}$' denotes the orthogonal complement in $\RR^d$.

\begin{lemma}
Let $n_1,\ldots,n_j$ be nonzero vectors in $\RR^d$ where $j\leq d$.  The following two statements are  equivalent.
\begin{enumerate}
\item For any $v\in\RR^d$, 
$$v\cdot n_1=v\cdot n_2=\cdots=v\cdot n_{j-1}=0\ \Longrightarrow \   v\cdot n_{j}=0.$$
\item $n_j=\sum_{k=1}^{j-1} \lambda_kn_k$, for some constants $\lambda_1,\ldots,\lambda_{j-1}$.
\end{enumerate}
\end{lemma}

\begin{proof}
 Write $n_j=v+w$ where $v\in\span\{n_1,\ldots,n_{j-1}\}$ and $w\in \span\{n_1,\ldots,n_{j-1}\}^{\perp}$.   Hence 
$n_j=\sum_{k=1}^{j-1} \lambda_kn_k +  w$, where $w\cdot n_k=0$ for all $k=1,\ldots,j-1$.  
  If (1.) holds, then also $w\cdot n_j=0$, and  
$$
0= w\cdot n_j = w\cdot \left(\sum_{k=1}^{j-1} \lambda_kn_k +  w\right) = w\cdot w.
$$
Thus $w=0$, which yields (2.).    We have proved (1.) $\Rightarrow$ (2.).

The implication (2.)\ $\Rightarrow$\  (1.) is trivial. 
\end{proof}

Recall that a hyperplane $H$ in $\RR^d$ is the zero set of an affine function (i.e., linear polynomial) 
\begin{equation}\label{eqn:hl} l(x) = l(x_1,\ldots,x_d) = \sum_{j=1}^d c_jx_j + b,\quad  H=\{x\in\RR^d\colon l(x)=0\}.
\end{equation}  
If $a\in H$ then $l(a)=0$ so that $b=-\sum c_ja_j$; hence  
\begin{equation} \label{eqn:normal}
 l(x) = n\cdot(x-a)
\end{equation}
 where $n=(c_1,\ldots,c_d)$.  For any other $p\in H$, $n\cdot(p-a)=l(p)=0,$ which is the well-known statement that $n$ is normal to $H$; by linearity, this also yields
\begin{equation}\label{eqn:ab}
n\cdot(x-a)    = n\cdot(x-p),
\end{equation}
so $a\in H$ is arbitrary.   

\begin{lemma} \label{lem:12}
Let distinct, non-parallel hyperplanes $H_1,H_2$ in $\RR^d$ be given as in (\ref{eqn:hl}) by the respective functions $l_1,l_2$.  Let $\eta\in\RR^d$ with $l_1(\eta)>0$ and $l_2(\eta)>0$.  Then there is $\zeta_1\in H_1$, $\zeta_2\in H_2$ such that 
\begin{enumerate}
\item $\eta\in I$ where  $I$ is the closed line segment  joining $\zeta_1$ to $\zeta_2$;

\item If $L$ denotes the line through $\zeta_1$ and $\zeta_2$, then
$$
I=\{x\in L\colon \min\{l_1(x), l_2(x)\}\geq 0\}.
$$ 
\end{enumerate}
\end{lemma}

\begin{proof}
Without loss of generality, we may assume $d=2$ by restricting to a plane that intersects $\{\eta\}\cup(H_1\cap H_2)$.  In the plane, we may suppose, following a possible translation, that $H_1$ and $H_2$ are lines through the origin, and  $l_1,l_2$ are the corresponding linear functions for which $H_j=\{x\in\RR^2\colon l_j(x)=0\}$.

Let $\zeta_1\in H_1$ satisfy $l_2(\zeta_1)=2l_2(\eta)$.    The line through $\zeta_1$ and $\eta$ is given by $t\mapsto t\zeta_1+(1-t)\eta$.  It intersects $H_2$ at a point $\zeta_2$ corresponding to the parameter $t=-1$:
$$
0=l_2(\zeta_2)=   l_2(t\zeta_1 +(1-t)\eta ) = tl_2(\zeta_1)+(1-t)l_1(\eta) = (1+t)l_2(\eta).  
$$
This gives the line segment $I$.  Another calculation yields $l_1(\zeta_2)=2l_1(\eta)$.  As one goes along $I$ from $\zeta_1$ to $\zeta_2$,  $l_1$ increases linearly from $0$ to $2l_1(\eta)$  while $l_2$ decreases linearly  from $2l_2(\eta)$ to $0$.  The second statement follows easily.
\end{proof}

\begin{lemma}\label{lem:1}
Let $H_1,\ldots,H_d$ be affine hyperplanes given by linear equations 
\begin{equation}\label{eqn:H} l_1(x)=0,\ \ldots,\  l_d(x)=0, \end{equation}  
and let $n_1,\ldots,n_d$ be the respective normal vectors as in (\ref{eqn:normal}).   
  Let $j\leq d$ and suppose the vectors $n_1,\ldots,n_j$ are linearly independent.  Then $H_1\cap\cdots\cap H_j$ is nonempty.  If $p$ is a point of the intersection, then $l_k(x)=n_k\cdot(x-p)$ for all $k=1,\ldots,j$ and 
 $$S:=\{x\in\RR^d\colon l_k(x)\geq 0 \hbox{ for all } k=1,\ldots,j\}$$ is a cone over $p$. 
\end{lemma}

\begin{proof}
As in (\ref{eqn:normal}), write $l_k(x)=n_k\cdot(x-a_k)$ where $a_k\in H_k$, $k=1,\ldots,j$.  By linear algebra, the hyperplanes have nonempty intersection, because if $x$ satisfies (\ref{eqn:H}) then it solves the linear system $Ax=c$, where $A$ is the matrix for which $A_{ik}$ is the $i$-th entry of $n_k$ and the entries of $c$ are given by $c_k:=n_k\cdot a_k$.  The hypothesis on the $n_k$s says that $A$ is of rank $j$; in particular, the system has a solution, $p$  say.  

Since $p\in H_j$ for each $j$, we have by the argument preceding (\ref{eqn:ab}) that
$$
S=\{x\in\RR^d\colon n_k\cdot(x-p)\geq 0 \hbox{ for all } k=1,\ldots,j\}.
$$
Thus $x\in S$ means that $x= v+p$ for some vector $v$ that satisfies $v\cdot n_j\geq 0$ for all $j$.  Then $\lambda v\cdot n_j\geq 0$ if $\lambda\geq 0$, so $x(\lambda):= p+\lambda v$ is also in $S$.  Thus if $S$ contains $x$, then it also contains the ray starting at $p$ and going through $x$.  Hence $S$ is a cone with vertex at $p$; in particular, we may write
$
S=\{p+\lambda v\colon \lambda\geq 0, v\in \hat S \}$,
where $\hat S = \{w\in\RR^d\colon w\in S, \|w-p\|=1\} .$
\end{proof}

\section{Simplices} \label{sec:simplices}

In this section we look at the geometry of simplices and characterize them by a generic linear independence condition on the supporting hyperplanes of their faces.  We then show that a compact convex polytope whose faces satisfy the condition may be generated by its supporting simplices. 

\begin{notation}\rm Given a set $S\subset\RR^d$, let $\co(S)$ denote its \emph{closed convex hull}, i.e., the smallest closed convex set containing $S$. \end{notation}

Recall that a $j$-dimensional simplex in $\RR^d$ is a set of the form $\co(\{p_0,\ldots,p_{j}\})$ for points in \emph{general linear position}, which means that $\{p_0,\ldots,p_j\}$ is not contained in a $j$-dimensional affine hyperplane.  Equivalently, the vectors $v_k:=p_k-p_l$,  $k\neq l$, are linearly independent, where $l\in\{0,\ldots,j\}$ is fixed.  In this section, we concentrate on simplices of full dimension ($j=d$), for which we will give a dual characterization by supporting hyperplanes, or equivalently, their affine defining functions.

Let $H_0,\ldots, H_{d}$ be hyperplanes defined by the affine functions $l_0,\ldots,l_d$, and define $L_j\colon\RR^d\to\RR^d$ for $j=0,\ldots,d$ by \begin{equation}\label{eqn:L}
L_j(x) := (l_0(x),\ldots,l_{j-1}(x),l_{j+1}(x),\ldots,l_d(x)).
\end{equation}
\begin{lemma}\label{lem:13} The following two conditions are equivalent for $j\in\{0,\ldots,d\}$.  
\begin{enumerate}
\item $L_j$ is invertible and $p_j=L_j^{-1}(0)$.
\item $\{p_j\} = \bigcap_{k\neq j}H_k$.  
\end{enumerate}

Suppose the above conditions hold for $j=j_1,j_2$.  Then $p_{j_1}\neq p_{j_2}$ if and only if $l_{j_1}(p_{j_1})\neq 0$.  (And if and only if $l_{j_2}(p_{j_2})\neq 0$, by symmetry.) 
\end{lemma}

\begin{proof}
The equivalence of (1.) and (2.) is elementary linear algebra, given that $l_k(p_j)=0$ for all $k\neq j$.  (This is the condition for a system of $d$ linear equations to have a unique solution.)  

If  $l_{j_1}(p_{j_1})=0$, then $L_{j_2}(p_{j_1})=0$, which says that $p_{j_2}=p_{j_1}$.  This gives the forward implication in the last statement by contraposition.  Conversely, if $l_{j_1}(p_{j_1})\neq 0$ then $L_{j_2}(p_{j_1})\neq 0 = L_{j_2}(p_{j_2})$, and hence $p_{j_1}\neq p_{j_2}$. 
\end{proof}

\begin{lemma}\label{lem:14}
Suppose the conditions of Lemma \ref{lem:13} hold for all $j=0,\ldots,d$, with $p_0,\ldots,p_d$ being the corresponding points.  
Either $p_j\neq p_k$ whenever $k\neq j$ (all points are distinct), or $p_0=p_1=\cdots=p_d$ (all points are equal).

 If all points are equal, then the set $S:= \{x\in\RR^d\colon l_j(x)>0 \hbox{ for all }j\}$ is either empty or an unbounded cone.

If all points are distinct, then each subcollection of points is in general linear position. Hence the closed convex hull of these points is a simplex.

\end{lemma}

\begin{proof}
Either all points are distinct, or at least two points are equal, say  $p_0=p_1$, which says that
$$
\{p_0\} = \bigcap_{k\neq 0} H_k = \bigcap_{k\neq 1} H_k =  \{p_1\}.   
$$
Taking the intersection with $H_1$, we compute for any $j\in\{1,\ldots,d\}$, that
\[ 
\{p_0\} = \bigcap_{k\neq 0} H_k = \bigcap_{k} H_k \subseteq \bigcap_{k\neq j} H_j =\{p_j\},  \]
and therefore $p_1=p_0=p_j$ for all $j$, i.e., all points are equal.  

If all points are equal, suppose without loss of generality that this point is the origin.  Then the $l_j$s are linear maps and for each $j$, 
$$
l_j(x)>0 \Rightarrow l_j(\lambda x) = \lambda l_j(x) >0 \hbox{ for all }\lambda>0.
$$
Hence $x\in S$ implies $\lambda x\in S$, so $S$ is either empty or an unbounded cone with vertex at the origin.

Suppose all points are distinct.  To verify, say, that $\{p_0,p_1,p_2\}$ are in general linear position, we show that $p_0-p_1$ and $p_2-p_1$ are linearly independent, i.e., 
\begin{equation}\label{eqn:c1c2}
c_1(p_0-p_1) + c_2(p_2-p_1)=0  
\end{equation}
implies $c_1=c_2=0$.  
Since $p_1\in H_0$, we have  $l_0(x) = n_0\cdot(x-p_1)$ for some vector $n_0$ normal to $H_0$ (see (\ref{eqn:ab})), and therefore, by Lemma \ref{lem:13}, 
$$
l_0(p_0)=n_0\cdot(p_0-p_1)\neq 0,\quad l_0(p_2)=n_0\cdot(p_2-p_1)=0.
$$
Take the inner product with $n_0$ on both sides of (\ref{eqn:c1c2}); then applying the above equations, we have $c_1n_0\cdot(p_0-p_1)=0$.  Hence $c_1=0$, and $c_2=0$ follows.

To show that larger subcollections of points are also in general linear position, we induct on the size of the set.  Considering say, $\{p_0,\ldots,p_j\}$, suppose
\begin{equation}\label{eqn:induction}
c_1(p_0-p_1) + c_2(p_2-p_1)+\cdots+c_j(p_j-p_1) = 0.
\end{equation}
As above, we have
$$
n_0\cdot(p_0-p_1)\neq 0, \ n_0\cdot(p_2-p_1)=0, \ \ldots, \ n_0\cdot(p_j-p_1)=0,
$$
and similarly, $c_1=0$.    Hence  equation (\ref{eqn:induction}) holds without the first term on the left-hand side.  By induction, the points of $\{p_1,\ldots,p_j\}$ are in general linear position, so $c_2=\cdots=c_j=0$.  
\end{proof}

In what follows we will assume that the hyperplanes $H_0,\ldots,H_d$ are given by $l_0,\ldots,l_d$ and satisfy:

\begin{itemize}
\item[($\star$)] {\it 
There are $d+1$ distinct points $p_0,\ldots,p_d$ such that $\bigcap_{k\neq j} H_k = \{p_j\}$ and $l_j(p_j)>0$ 
 for each $j=0,\ldots,d$.}
\end{itemize}

\begin{remark} \rm The statement $l_j(p_j)>0$ in  ($\star$) may be obtained from $l_j(p_j)\neq 0$ in the final conclusion of Lemma \ref{lem:13}. We can ensure positivity by possibly replacing $l_j$ with $-l_j$; this does not change  $H_j$.

Also, if condition 1 of Lemma \ref{lem:13} holds for all $j$, then {any collection of $d$ normal vectors (to $d$ hyperplanes) forms a  linearly independent set}.
 \end{remark}

\begin{proposition}\label{prop:14}
Suppose the conditions of Lemma \ref{lem:13} hold, with the $l_j$s satisfying ($\star$).  Then 
$$
\co(\{p_0,\ldots,p_d\}) = \{x\in\RR^d\colon l_j(x)\geq 0 \hbox{ for all } j=0,\ldots,d\}.
$$
\end{proposition}

\begin{proof}
Denote by $L$ the set on the left and by $R$ the set on the right.  Note that for any $x,y\in\RR^d$ we have 
\begin{equation}\label{eqn:lem15} l_j(tx+(1-t)y) = tl_j(x) + (1-t)l_j(y), \hbox{ for each } j.\end{equation}
Hence the left-hand side of the above is nonnegative if $l_j(x),l_j(y)\geq 0$ and $t\in[0,1]$.  It follows that $R$ is convex. Also, $R$ contains $\{p_0,\ldots,p_d\}$ by ($\star$).  Hence $L\subseteq R$.

Conversely,  to show $R\subseteq L$ we do an inductive argument on the dimension $d$.  If $d=1$ then we have two points $x_0,x_1$ and corresponding affine functions $l_0,l_1$.  Using (\ref{eqn:lem15}) it is easy to verify that $R$ is the line segment joining $x_0$ to $x_1$, which coincides with $L$. (So $R=L$ in this case.)

Now consider higher dimensions, $d>1$.  Let $a\in R$, so that $l_j(a)\geq 0$ for all $j$.  Since we need to show that $a\in L$, assume $a\not\in\{p_0,\ldots,p_d\}$ (otherwise $a\in L$ by convexity of $R$ and we are done).  For the purpose of induction, assume the inclusion $R\subseteq L$ holds in dimension less than $d$.  

If $l_k(a)=0$ for some $k$, then $a\in H_k$.  Also, by definition, $p_j\in H_k$ for all $j\neq k$.  Since the hyperplane $H_k$ is an affine space of dimension $d-1$, induction yields 
\[ \begin{aligned}
a &\in \{x\in\RR^d\colon l_j\geq 0 \hbox{ for all } j=0,\ldots d, \ j\neq k \}\cap H_k  \\
&\subseteq  \co(\{p_0,\ldots,p_{k-1},p_{k+1},\ldots, p_{d}\})\cap H_k \subseteq \co(\{p_0,\ldots,p_{d}    \}) \cap H_k  \subseteq L,
\end{aligned} \]
i.e., $a\in L$.

On the other hand, suppose $l_j(a)>0$ for all $j$.  By Lemma \ref{lem:12} there is a line segment $I$ through $a$ that intersects two of the hyperplanes  ($H_j,H_k$ say).  By possibly shrinking the line segment, we may assume that $I\subset R$ with endpoints $\zeta_j\in (H_j\cap\partial R)$ and $\zeta_k\in (H_k\cap\partial R)$.  Now $l_j(\zeta_j)=0$ so we may apply the argument in the previous paragraph to show that $\zeta_j\in L$.  Similarly, $\zeta_k\in L$.  By convexity, $I\subset L$, hence $a\in L$.
\end{proof}



\begin{lemma}\label{lem:119}
Suppose condition ($\star$) holds.  For each $j=0,\ldots,d$, write $l_j(x) = n_j\cdot(x-a_j)$ for some $a_j\in H_j$ and normal vector $n_j$.  Let $S$ denote the set in Proposition \ref{prop:14}. 
Then given $x\in \intt(S)$ and $j\in\{0,\ldots,d\}$,  
\begin{equation}\label{eqn:lem19}
n_j\cdot(x-p_j)<0,  \hbox{ and } n_k\cdot(x-p_j)>0 \hbox{ for all } k\neq j.
\end{equation}
\end{lemma}

\begin{proof}
The ray starting at $p_j$ and going through $x$ must intersect the boundary of $S$ at a point $\zeta$, since $S$ is a bounded convex set.  We claim that $\zeta\in H_j$.   Otherwise, 
 $\zeta\in H_k$ ($k\neq j$), and  because $p_j\in H_k$, we must have $x\in H_k$ by convexity, i.e., $l_k(x)=0$, which contradicts the fact that $l_k(x)>0$.  So $\zeta\in H_j$, and   
$l_j(p_j)>0=l_j(\zeta),$ 
which means that $l_j$ decreases linearly along the ray.  It follows that $0>l_j(x)-l_j(p_j) = n_j\cdot(x-p_j)$.  

If $k\neq j$ then $p_j\in H_k$, and therefore $n_k\cdot(x-p_j)=l_k(x)>0$.
\end{proof}

\begin{lemma}\label{lem:3.7}
Let $l_0,\ldots,l_d$ be affine maps, $l_j(x)=n_j\cdot(x-a_j)$ for each $j$, such that any choice of $d$ vectors in $\{n_0,\ldots,n_d\}$ is linearly independent.  For each $j$, let $p_j$ be the solution to the system of equations $l_k(x)=0$ for all $k\neq j$.    Let  $$S:=\{x\in\RR^d: l_j(x)\geq 0 \hbox{ for each }j\}.$$    Suppose there exists $x\in S$ and $j\in\{0,\ldots,d\}$ such that  (\ref{eqn:lem19}) holds.  Then the points $p_0,\ldots,p_d$ are distinct, condition $(\star)$ holds for these points, and $S=\co(\{p_0,\ldots,p_d\})$.   
\end{lemma}

\begin{proof}
Suppose $j\in\{0,\ldots,d\}$ and $x\in S$ with $n_j\cdot(x-p_j)<0$.  This says that $l_j(x)-l_j(p_j)<0$, so $l_j(p_j)>l_j(x)>0$.  On the other hand, if $k\neq j$ then $l_j(p_k)=0$, and hence $p_j\neq p_k$.  So the points $\{p_0,\ldots,p_d\}$ are distinct by Lemma \ref{lem:14}.  

For any other $k\neq j$, we have $l_k(x)>0$.  Since $S$ is convex, the closed line segment $I$ joining $p_k$ to $x$ is contained in $S$, so $l_k\geq 0$ on any point of $I$.  If $l_k(p_k)=0$ then $p_k\in H_k$ and therefore 
$$
\bigcap_{\nu\neq k} H_{\nu} =\{p_k\} = \bigcap_j H_j \subseteq \bigcap_{\nu\neq j} H_{\nu} = \{p_j\},
$$
so $p_k=p_j$, contradicting the previous paragraph.  Therefore $l_k(p_k)>0$.

Thus $l_j(p_j)>0$ for all $j$, so condition ($\star$) holds and $S=\co(\{p_0,\ldots,p_d\})$ by Proposition \ref{prop:14}.  
\end{proof}

\begin{remark} \label{rem:210}  \rm
Condition (\ref{eqn:lem19}) on the normal vectors $n_j$  can be observed geometrically.  Let $K\subset\RR^2$ be the quadrilateral $K=\{ x\in\RR^2\colon \ell_j(x)\geq 0\}$   given by
$$\ell_1(x)=x_1,\ \ell_2(x)=x_2,\ \ell_3(x) = 3-3x_1-x_2,\  \ell_4(x) = 3-x_1-3x_2. $$
Let $S_j=\{x\in\RR^2\colon \ell_k(x)\geq 0 \hbox{ for } k\neq j\}$.  The normal vectors associated to the $\ell_j$s that define  $S_1$, $S_2$, $S_3$, and $S_4$ respectively are pictured below.

\includegraphics[height=3cm]{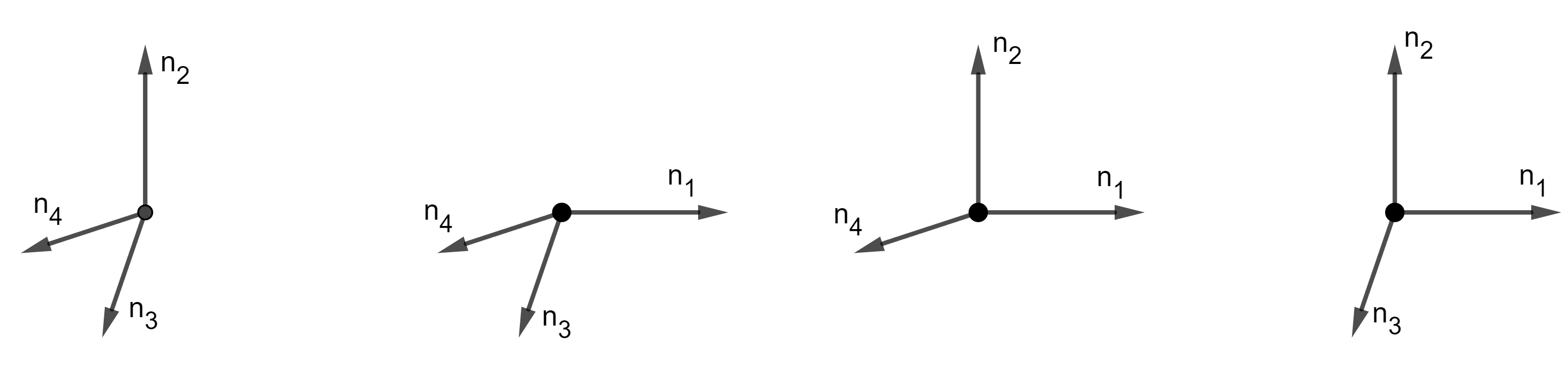}

Here $d=2$ in which each triangle is given by 3 lines.  Then condition (\ref{eqn:lem19}) says the following:  for any two normal vectors, the (positive) cone generated by these vectors never contains the third vector.    Clearly the normal vectors to the edges of $S_3,S_4$ have this property, but those for $S_1,S_2$ do not.  By Lemmas \ref{lem:119} and \ref{lem:3.7}, only  $S_3,S_4$ form a simplex (i.e., a triangle) containing $K$.  The sets $S_1,S_2$ are (unbounded) cones.
\end{remark}

\begin{theorem} \label{thm:simplices}
Let $N\geq d$ and let $K$ be a compact convex polytope of dimension $d$ in $\RR^d$ with $N+1$ faces $F_0,\ldots,F_N$ of codimension 1.  Let $H_0,\ldots, H_N$ be the hyperplanes containing these faces, $F_j\subset H_j$ for each $j=0,\ldots,N$. Let $n_0,\ldots,n_N$ be the corresponding normal vectors.  

If each collection of $d$ vectors in $\{n_0,\ldots,n_N\}$ is linearly independent, then $K$ is an intersection of at most finitely many $d$-dimensional simplices
$$
K = \bigcap_{j=1}^M S_j,\quad \hbox{where } M\leq \binom{N+1}{d+1},
$$ 
such that each face of $S_j$ contains a face of $K$. 
\end{theorem}

\begin{proof}
Choose faces $F_1,\ldots,F_d$ with corresponding hyperplanes $H_1,\ldots,H_d$ and linear maps $l_1,\ldots,l_d$.  Let $\{p_0\}:=\bigcap_{j=1}^dH_j$.  Then the set 
$$
C:=\{x\in\RR^d\colon l_j(x)\geq 0 \hbox{ for all } j=1,\ldots,d\}
$$
is a closed convex cone containing $K$ with vertex at $p_0$.

  Choose a vector $v$ such that the ray $\{p_0+tv\colon t\in[0,\infty)\}$ goes into the interior of $C$ and meets points in the interior of $K$.  Since $K$ is compact, there is a $\tau\in(0,\infty)$ such that $\zeta:=p_0+\tau v\in\partial K$ and $p_0+tv\not\in K$ for all $t>\tau$.  Now $\zeta$ is contained in a face of $K$, say $\zeta\in F_0$, with supporting hyperplane $H_0$ and affine map $l_0$.  We have   
\begin{equation}\label{eqn:thm111a}
l_0(p_0+tv) \left\{ \begin{aligned} 
&>0\hbox{ if } t<\tau,\\
&=0 \hbox{ if } t=\tau, \\
& <0 \hbox{ if }t>\tau. \end{aligned}\right.
\end{equation}

Let $x_0$ be the midpoint of the segment joining $p_0$ and $\zeta$.    Then by (\ref{eqn:thm111a}), 
 $l_0$  decreases linearly along the segment in the direction of $\zeta-p_0$, or equivalently, the direction of $x_0-p_0$.  Hence 
$$
l_0(\zeta)=0<l_0(x_0)<l_0(p_0) \Longrightarrow n_0\cdot(x_0-p_0)<0,
$$
where as before, $n_0$ is the normal vector to $H_0$ for which $l_0(x)=n_0\cdot(x-a),a\in H_0$. 

If $j\in\{1,\ldots,d\}$, then $l_j$ increases linearly as we go along the ray $p_0+tv$ into the interior of $C$.  Therefore
$$
l_j(p_0)=0<l_j(x_0)<l_j(\zeta) \Longrightarrow n_j\cdot(x_0-p_0)>0.
$$

By Lemma \ref{lem:3.7}, 
\begin{equation}\label{eqn:simplexl}
S:=\{x\in\RR^d\colon l_j(x)\geq 0 \hbox{ for each } j=0,\ldots,d\} =\co(\{p_0,\ldots,p_d\}), 
\end{equation}
where $\{p_k\}=\bigcap_{j\neq k}H_j$.  By Lemma \ref{lem:14}, $S$ is a $d$-dimensional simplex.  By construction, $l_j(x)\geq 0$ for each $x\in K$ and $j=1,\ldots,d$, so $S\subset K$.

In the above argument the selection of faces $F_1,\ldots,F_d$ (and their supporting hyperplanes) was arbitrary.  Hence every supporting hyperplane that contains a face of $K$ is also a hyperplane that contains the face of some simplex containing $K$.  Now every point $a\not\in K$ satisfies $l(a)<0$ for some affine map that gives the supporting hyperplane $H$ of a face of $K$.  By (\ref{eqn:simplexl}) it is also in the complement of a constructed simplex $S\supseteq K$.  
Hence $K$ must be exactly the intersection of such simplices.  Finally note that $K$ cannot be an intersection of more than $\binom{N+1}{d+1}$ distinct simplices, which is the number of ways to choose $d+1$  faces  of $K$.
\end{proof}

\begin{remark}\rm
If $S_j,S_k$ are simplices as in the above theorem and $S_j\subset S_k$, then $S_j=S_k$.  To see this, note that if a codimension 1 face $F_{S_k}$ of $S_k$ contains a codimension 1 face $F_K$ of $K$, then $F_{S_k}\supset F_{S_j}\supset F_K$ where $F_{S_j}$ is a codimension 1 face of $S_j$.  Both $F_{S_j},F_{S_k}$ are contained in the hyperplane containing $F_K$.  We conclude that every hyperplane containing a codimension 1 face of $S_j$ also contains a codimension 1 face of $S_k$ and vice versa.  There are $d+1$ such hyperplanes, which must characterize both $S_j$ and $S_k$ as simplices.  So $S_k=S_j$.  
\end{remark}

\begin{figure}
\includegraphics[height=4cm]{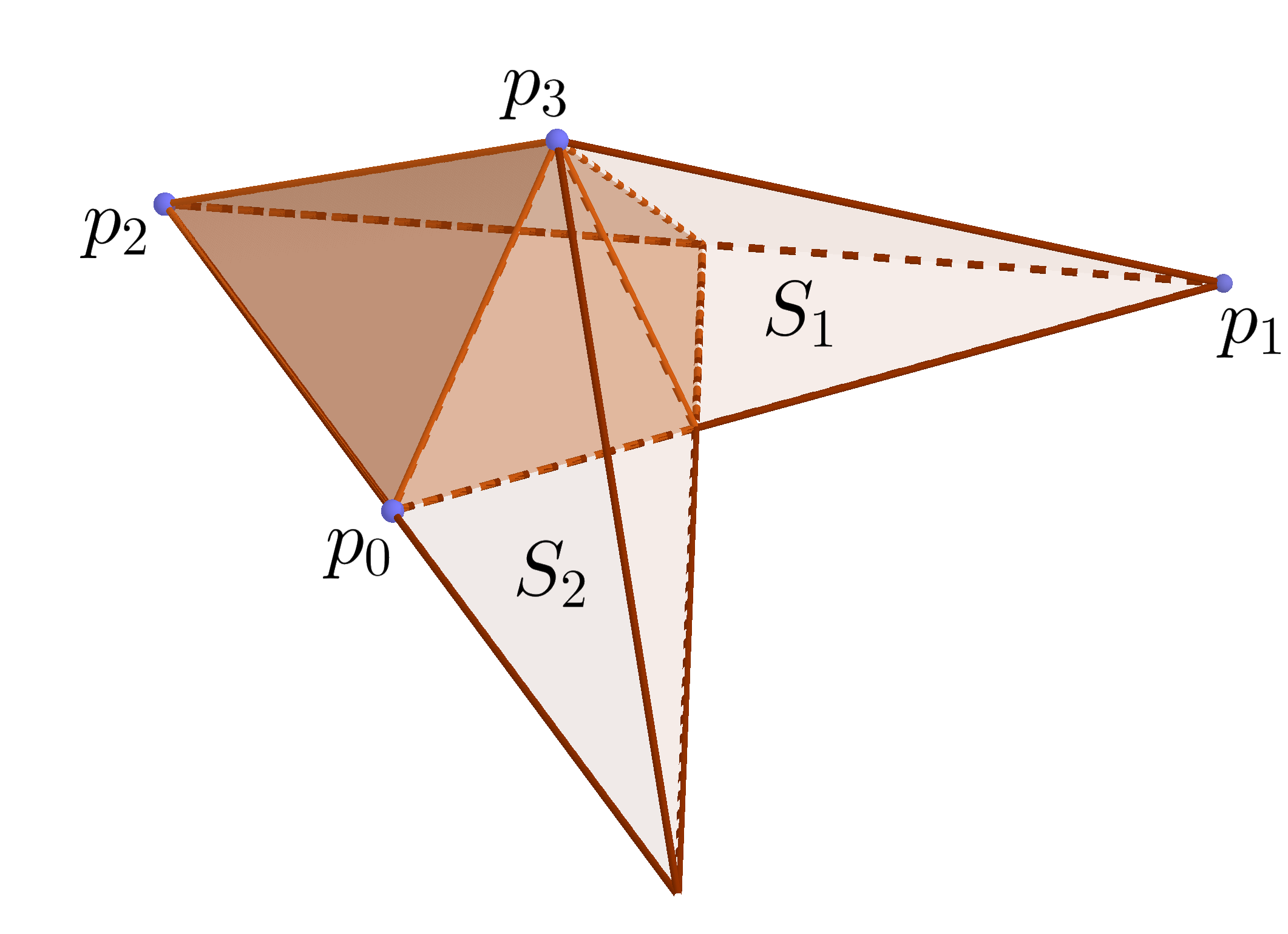}
\caption{Theorem \ref{thm:simplices} in $\RR^3$, with  $K=S_1\cap S_2$ and $S_1=\co\{p_0,p_1,p_2,p_3\}$.}
\end{figure}

Figure 2 illustrates a polytope in $\RR^3$ whose faces satisfy the hypotheses of Theorem \ref{thm:simplices}, and its associated simplices.  (We will consider more general polytopes in the next section where we discuss \emph{strips}.)  The simplices minimally support $K$ in the following sense.

\begin{lemma}\label{lem:simpcirc}
Let $S$ denote one of the simplices $S_j$ in the previous theorem.  For any vector $b\in\RR^d$, the translate $b+K$ contains points not in $S$.
\end{lemma}

\begin{proof}
Let $b$ be a vector. The simplex $S$ has a supporting hyperplane, say $H_0$, for which $b$ is not parallel to $H_0$.
Let $a_0\in H_0\cap \partial K (\subset\partial S)$.  There exists $x\in K$ and $\epsilon>0$ such that either $\epsilon b=a_0-x$ or $\epsilon b=x-a_0$.  In the first case, by convexity of $S$, 
$$
x\in S, x+\epsilon b \in\partial S  \   \Rightarrow \   a_0+b = x + (1+\epsilon) b\not\in S.
$$
In the second case, take another supporting hyperplane to $S$, say, $H_1$, with $b$ not parallel to $H_1$.  Let $p_1$ denote the vertex of $S$ that does not lie on $H_1$.    By Lemma \ref{lem:119} (equation (\ref{eqn:lem19})), $n_1\cdot(x-p_1)<0$, so that
\begin{eqnarray*}
0 > n_1\cdot(x-p_1)  
= n_1\cdot( x-a_0) + n_1\cdot(a_0-p_1) 
&=& n_1\cdot(x-a_0) + l_1(a_0) \\
&\geq & n_1\cdot (x-a_0)    \\
&=& \epsilon n_1\cdot b.
\end{eqnarray*}
If we now pick $a_1\in H_1\cap  \partial K$, then
$$
0>\epsilon n_1\cdot b = l_1(a_1+\epsilon b). 
$$
Since $S\subset\{l_1\geq 0\}$,   $a_1+\epsilon b\not\in S$.   Hence by convexity of $S$, $a_1+b\not\in S$.
\end{proof}

\section{Strips} \label{sec:strips}

We will define a \emph{(convex) strip} in $\RR^d$ to be the preimage  $L^{-1}(A)$ of a compact convex set  $A\subset\RR^j$ of dimension $j$, where $j\in\{1,\ldots,d\}$ and $L\colon\RR^d\to\RR^j$ is a linear map of rank $j$.  By rotating coordinates, a strip can be put into the form
$$
S = A\times \RR^{d-j} \subset \RR^j\times\RR^{d-j} = \RR^d;
$$
the set $A$ is called the \emph{($j$-dimensional) orthogonal cross-section}.   

An infinite (in both directions) prism with a polygonal base is an example of a strip.  
Figure 3 illustrates supporting strips in $\RR^3$.

\begin{figure}
\begin{multicols}{2}
\includegraphics[height=4cm]{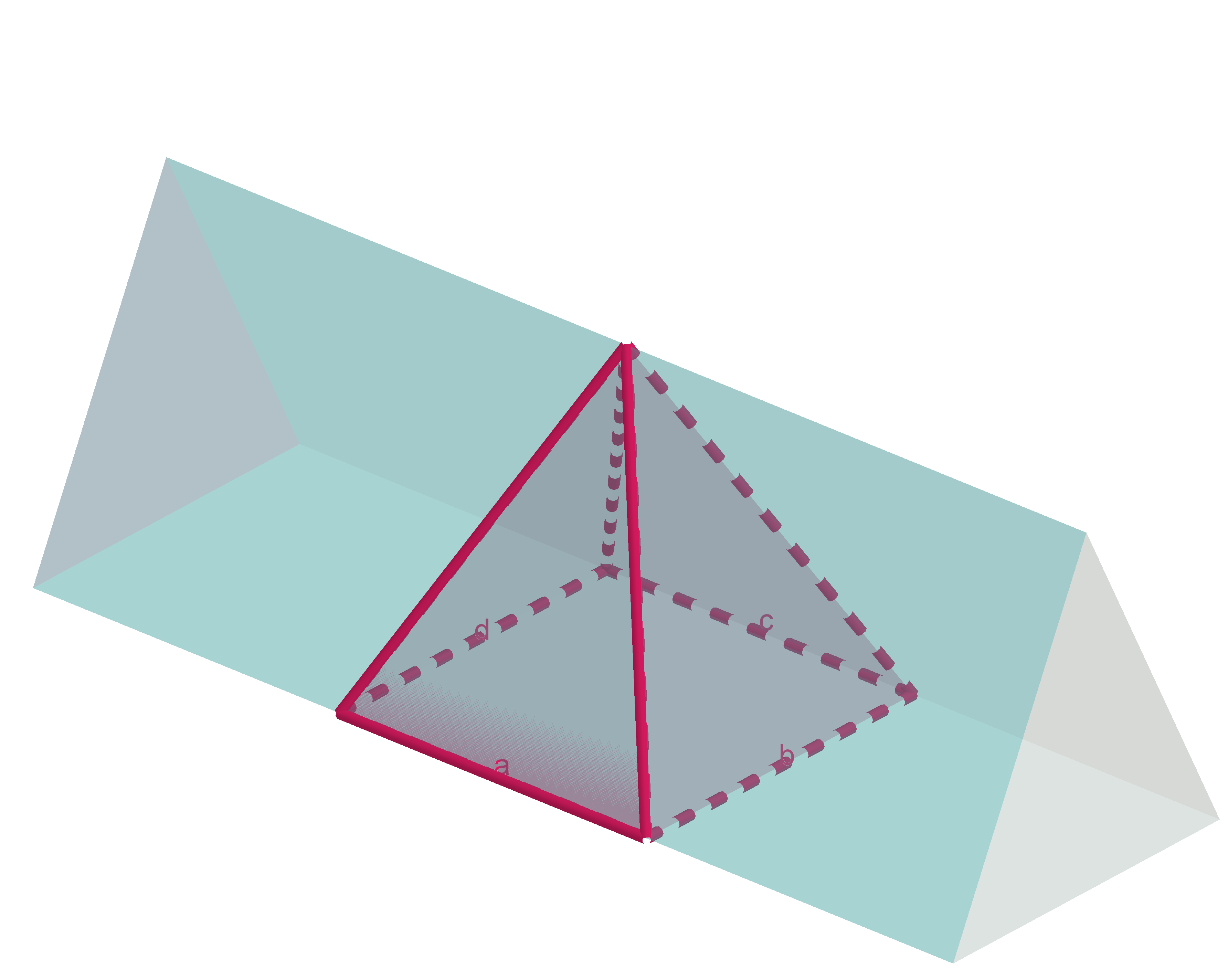}
\includegraphics[height=4cm]{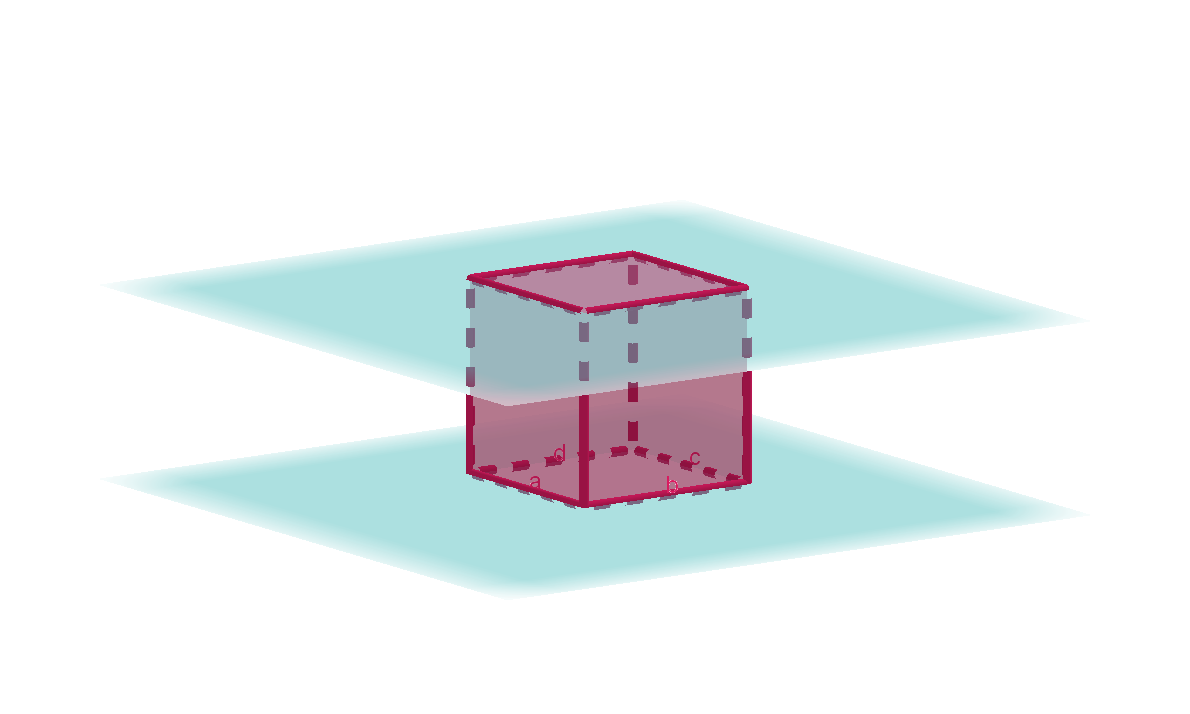}
\end{multicols} \caption{Examples of strips in $\RR^3$.}
\end{figure}


\begin{definition}\rm
Let $K\subset\RR^d$ be a convex set, $K\subseteq S$ where $S$ is a strip.  Then $S$ is a \emph{supporting strip} (to $K$) if for any translation $b$ in a direction parallel to its orthogonal cross-section, $b+K\not\subset S$.
\end{definition}

\begin{lemma}
$S\supset K$ is a supporting strip if and only if there does not exist a translation $b\in\RR^d$ such that $b+K\subset\intt(S)$.
\end{lemma}

\begin{proof}
Decompose $b$ into the orthogonal sum $b=b_1+b_2$ where $b_1$ is along the strip $S$ and $b_2$ is along the orthogonal cross-section.  The geometric properties of $b+K$ in relation to $S$ are the same as the geometric properties of $b_2+K$ in relation to $-b_1+S$.  But $-b_1+S=S$.  Thus $b+K\in\intt(S)$ if and only if $b_2+K\in\intt(S)$, in which case $S$ cannot be a supporting strip.
\end{proof}

To obtain supporting simplices in Theorem \ref{thm:simplices}, the normal vectors for any collection of $d$ supporting hyperplanes were required to be linearly independent.  If only $j<d$ normal vectors in some collection are linearly independent, we get a strip.

\begin{proposition}\label{prop:114}
Suppose $H_0,\ldots,H_j$ are supporting hyperplanes to a convex polytope $K\subset\RR^d$, where $j<d$.  Let $l_k(x) = n_k\cdot(x-a_k)$, $a_k\in H_k$, be the corresponding linear maps for $k=0,\ldots, j$.  
Suppose the set of normal vectors $N:=\{n_0,\ldots,n_j\}$ satisfies the following two conditions:
\begin{itemize}
\item $\span(N)$ has dimension $j$ and every subset of $N$ of size $j$ is linearly independent;
\item There exists $x\in K$ and $p_0\in \bigcap_{k=1}^j H_k$ such that 
\begin{equation}\label{eqn:simpcondition}
n_0\cdot(x-p_0)<0 \hbox{ and } n_k\cdot(x-p_0)>0 \hbox{ for all } k\in\{1,\ldots,j\}.
\end{equation}
\end{itemize}
  Then there is a supporting strip to $K$ that is equivalent under rotation to   $\Sigma\times\RR^{d-j}$, where $\Sigma\subset\RR^j$ is a $j$-dimensional simplex.  
\end{proposition}

\begin{proof}
By rotating coordinates, we may assume without loss of generality that
$$
\span(N) = \{x\in\RR^d\colon x_{j+1}=\cdots=x_d=0\},
$$ 
so that $n_k=(n_{k1},\ldots,n_{kj},0,\ldots,0)$ for all $k$.  Hence 
\begin{eqnarray*}
l_k(x)= n_{k1}x_1 + \cdots + n_{kj}x_k + c_k &=& n_k'\cdot x' + c_k \\ & =:& \tilde l_k(x')
\end{eqnarray*}
where $c_k=-n_k\cdot a_k$ and $n_k',x'$ are the projections of $n_k,x$ to the first $j$ coordinates.

Define 
\begin{equation*} 
 S:=  \{x\in\RR^d \colon l_k(x)\geq 0 \hbox{ for all } k=0,\ldots,j\}.\end{equation*}
  Then we have 
$$
S = \{x'\in\RR^j\colon \tilde l_k(x')\geq 0\hbox{ for all } k=0,\ldots,j \} \times\RR^{d-j} 
=: \Sigma\times\RR^{d-j}. 
$$
Since $n_{k(j+1)}=\cdots=n_{kd}=0$, equation (\ref{eqn:simpcondition}) says that 
$$
n_0'\cdot(x'-p_0')<0 \hbox{ and } n_k'\cdot(x'-p_0')>0 \hbox{ for all } k\neq 0.
$$
By Lemmas \ref{lem:3.7} and \ref{lem:14} applied in dimension $j$,  $\Sigma$ is a $j$-dimensional simplex. 

Finally, we verify that $S$ is a supporting strip.  Since we are only considering translations parallel to the cross-section $\Sigma$, it is sufficient to show that $\Sigma$ is a simplex that minimally supports the projection 
$$
K':=\{x' \colon x\in K\}\subset\RR^j.
$$
But this is true by Lemma \ref{lem:simpcirc}.
\end{proof}

  Before proving the main result of this section, we state an elementary lemma. 
\begin{lemma}\label{lem:4.4}
Suppose the set of vectors $\{n_0,n_1,\ldots,n_k\}$ is linearly dependent, but $\{n_1,\ldots,n_k\}$ is linearly independent.  Then there is $j\leq k$ and $j$ elements, say, $n_{k_1},\ldots,n_{k_j}$, such that the set  
$A:=\{n_0,n_{k_1},\ldots,n_{k_j}\}$ is linearly dependent, but any $j$ elements of $A$ are linearly independent.
\end{lemma}

\begin{proof}
By hypothesis, the collection $\mathcal{D}\subset\mathcal{P}(\{n_0,\ldots,n_k\})$ of subsets that contain linearly dependent elements is a non-empty collection, and each of them contains $n_0$.   Now choose $A\in\mathcal{D}$ with the smallest number of elements.
\end{proof}

\begin{theorem} \label{thm:simpliciesstrips}
Let $K$ be a compact convex polytope of dimension $d$ in $\RR^d$ with $N$ faces.  Then $K$ is a finite intersection
$$
K=\bigcap_{j=1}^M S_j  
$$
where each $S_j$ is either a $d$-dimensional simplex or a strip whose orthogonal cross-section is a lower-dimensional simplex, and each face of $S_j$ contains a face of $K$.
\end{theorem}

\begin{proof}
Choose a face, say $F_0$, of $K$, and vertex $p_0\not\in F_0$.  An examination of the proof of Theorem \ref{thm:simplices} shows that there are supporting hyperplanes $H_0,\ldots,H_d$ of $K$, such that $H_0$ contains $F_0$, $\{p_0\}= \cap_{k=1}^d H_k$, and for some $x\in\intt(K)$, 
$$
n_0\cdot(x-p_0)<0,\ n_k\cdot(x-p_0)>0 \hbox{ for all } k=\{1,\ldots,d\},
$$
where $n_0,\ldots,n_d$ are the corresponding normal vectors.  Thus condition (\ref{eqn:lem19}) holds with $j=0$.  

If each subset of  $\{n_0,\ldots,n_d\}$ of size $d$ is linearly independent, then, as in  Theorem \ref{thm:simplices},
\begin{equation}\label{eqn:simplex}
S:=\{x\in\RR^d\colon l_k(x)\geq 0 \hbox{ for all } k=0,\ldots,d \}
\end{equation}
is a supporting simplex of $K$.  

Otherwise, note that by construction $\{n_1,\ldots,n_d\}$ is linearly independent, since the corresponding hyperplanes $H_1,\ldots,H_d$ intersect at a point.    In view of the previous lemma, $\{n_k\}_{k=0}^j$, say, is linearly dependent for some $j<d$,  but each subset of size $j$ is linearly independent.  Then 
(\ref{eqn:simpcondition}) holds, and by Proposition \ref{prop:114}, the set 
\begin{equation}\label{eqn:strip}
S:=\{x\in\RR^d\colon l_k(x)\geq 0 \hbox{ for all }k=0,\ldots,j \}
\end{equation}
is a supporting strip of $K$ whose cross-section is a $j$-dimensional simplex.

Since $K$ has a finite number of faces, there are at most a finite number of distinct supporting simplices and strips $S_j$, with the property that each face of $S_j$ contains a face of $K$.
\end{proof}

\section{Pluripotential theory} \label{sec:pluripotential}

Let $\calL=\calL(\CC^d)$ denote the \emph{Lelong class}, which is the class of plurisubharmonic (psh) functions on $\CC^d$ of at most logarithmic growth: $u\in\calL$ if there exists $C\in\RR$ such that
$$
u(z) \leq \log^+|z| +C \hbox{ for all }z\in\CC^d,
$$
where $\log^+|z|=\max\{0,\log|z|\}$.  The class $\calL^+=\calL^+(\CC^d)$ is given by those functions $u\in\calL$ for which there also exists $c\in\RR$ such that the lower bound
$$
u(z)\geq \log^+|z| +c \hbox{ for all }z\in\CC^d
$$ holds.  

Recall that the (Siciak-Zaharjuta) extremal function is defined by (\ref{eqn:szextremal}) as an upper envelope of polynomials.  A theorem of Siciak and Zaharjuta says that this function is also given as an upper envelope of functions in $\calL$: for $K\subset\CC^d$ compact,
\begin{equation}\label{eqn:vkdef}
V_K(z) = \sup\{ u(z)\colon u\in\calL, \ u\leq 0\hbox{ on } K\}.
\end{equation}
Clearly $V_{K_1}\leq V_{K_2}$ if $K_1\supseteq K_2$, by definition.   

 The \emph{upper semicontinuous regularization of $V_K$} is $V_K^*(z):= \limsup_{s\to z} V_K(s)$.

\medskip

 We recall below some properties of the extremal function in $\CC^d$ that we will need; for proofs of the following, see \cite{klimek:pluripotential}, chapter 5.

\begin{proposition}   \label{prop:61} 
\begin{enumerate}
\item Let $d=d_1+d_2$ and $z=(z',z'')$ in $\CC^d=\CC^{d_1}\times\CC^{d_2}$.   Suppose $K_1\subset\CC^{d_1}$ and $K_2\subset\CC^{d_2}$ are compact.  Then $$V_{K_1\times K_2}(z',z'') = \max\{V_{K_1}(z'),V_{K_2}(z'')\}.$$
\item 
 Suppose $P=(P_1,\ldots,P_d)\colon\CC^d\to\CC^d$ is a polynomial mapping of degree $k$ with the property that $\hat P^{-1}(0) = \{0\}$, where $\hat P=(\hat P_1,\ldots,\hat P_d)$ and $\hat P_j$ denotes the homogeneous part of $P_j$ of degree $k$.  Then for any compact $K\subset\CC^d$,
$$
V_{P^{-1}(K)}(z) = \tfrac{1}{k}V_K(P(z)).
$$
\item Suppose $K_1\supset K_2\supset\cdots$ is a sequence of compact sets in $\CC^d$ decreasing to a compact set $K=\bigcap_j K_j$.  Then
$V_{K_j}(z)\nearrow V_K(z)$ for all $z\in\CC^d$.  
\item The following are equivalent:\begin{enumerate}[(i)] \item $V_K=V_K^*$;  (ii) $V_K$ is continuous; and (iii) $V_K^*(z)=0$ for all $z\in K$. \end{enumerate} 
\end{enumerate} \qed
\end{proposition}


A set $K$ for which any of the equivalent conditions in part 4 hold is said to be  \emph{regular}.  A real ball and simplex are regular, as can be observed from the explicit formulas of their extremal functions in Section \ref{sec:explicit}.  As a consequence:

\begin{lemma}\label{lem:reg}
A convex body $K\subset\RR^d$ is regular. 
\end{lemma}

\begin{proof}
Let $a\in K$.  Choose $d$ additional points $b_1,\ldots,b_d$ in $K$, in general linear position, so that the convex hull $\text{co}(a,b_1,\ldots,b_d)=:S$ is a $d$-dimensional simplex in $K$.  Then $V_K\leq V_S$, so $V_K^*\leq V_S^*$.  Since $S$ is regular and $a\in S$, $0= V_S^*(a) \geq V_K^*(a)$ by regularity condition (iii).  Hence $V_K^*(a)=0$, and since $a$ was arbitrary, condition (iii)  holds for all points in $K$.  So $K$ is regular.  
\end{proof}

Equation (\ref{eqn:vkdef})  also makes sense for a possibly unbounded set $S$: define  
$$V_S(z):=\sup\{u(z)\colon u\in \calL,\ u\leq 0 \hbox{ on } S\}.$$   
 We are interested in the special case in which $S$ is  a strip. 

\begin{lemma}\label{lem:62}
Let $K\subset\RR^j$ be a compact set and $S=\pi^{-1}(K)\subset\RR^d$, where $\pi:\RR^d\to\RR^j$ ($j<d$) is a linear map onto $\RR^j$.   Then $V_S(z)=V_K\circ\pi(z)$ for all $z\in\CC^d$.
\end{lemma}

\begin{proof}
If $z\in S$ then clearly $V_S(z)=V_K(\pi(z))=0$.  Also, as $|z|\to\infty$, 
$$
V_K(\pi(z)) - \log|z| = (V_{K}(\pi(z))-\log|\pi(z)|) + (\log|\pi(z)|-\log|z|) = O(1) + O(1), 
$$
so $V_K\circ\pi\in\calL$ and therefore $V_K\circ\pi\leq V_S$.  

It remains to show that $V_K\circ\pi\geq V_S$ for $z\not\in S$.  First, composing with a linear map and using Proposition \ref{prop:61}(2), we can assume that 
$$\pi(z) = \pi(z_1,\ldots,z_d) = (z_1,\ldots,z_j)=:z'$$ and $S=\pi^{-1}(K)= K\times\RR^{d-j}$ for some  compact  $K\subset\RR^j$.   Write $z''=(z_{j+1},\ldots,z_d)$ for the remaining coordinates, so $z=(z',z'')$. 

  Let $z\not\in S$. Then $z\not\in K\times B_R$ for any $R>0$, where 
$B_R=\{z''\in\RR^{d-j}\colon |z''|\leq R\}$.  Also, 
\begin{equation}\label{eqn:R}
V_S(z) = V_{K\times\RR^{d-j}}(z) \leq V_{K\times B_R}(z) = \max\{ V_K(z'),V_{B_R}(z'') \} \hbox{ for any }R>0
\end{equation}
by Proposition \ref{prop:61}(1).  

Using the linear map $z''\mapsto z''/R$, and applying Proposition \ref{prop:61}(2),    we have 
$$V_{B_R}(z'')=V_{B_1}(z''/R) \longrightarrow V_{B_1}(0)=0  \hbox{ as } R\to\infty,  $$
by continuity of $V_{B_1}$ and the fact that $0\in B_1$.  Taking the limit as $R\to\infty$ in (\ref{eqn:R}), 
$$
V_S(z)\leq \lim_{R\to\infty} \max\{V_K(z'),V_{B_R}(z'')\} = \max\{V_K(z'),0\} = V_K(z') = V_K(\pi(z)),
$$
which completes the proof. 
\end{proof}

\section{Extremal functions for a real ball and simplex}\label{sec:explicit}

Let $S=\co(p_0,\ldots,p_d)$ be a simplex in $\RR^d$, where the points $p_j$ are in general linear position.  
Associated to $S$ are \emph{barycentric coordinates} which are given by
$$S \ni  z\longmapsto \lambda=(\lambda_0,\ldots,\lambda_d) \in \RR^{d+1}  
$$
such that for any point $z\in S$, the components $\lambda_0=\lambda_0(z),\ldots,\lambda_d=\lambda_d(z)$ satisfy 
\begin{eqnarray}
\lambda_j\in[0,1] &\ & \hbox{for all }j=0,\ldots,d,   \nonumber  \\
(\lambda_0+\cdots+\lambda_d)   z &=& \lambda_0p_0+\cdots+\lambda_dp_d,    \label{eqn:61}  \\  
\lambda_0+\cdots+\lambda_d &=& 1 .   \label{eqn:62}
\end{eqnarray} 
Barycentric coordinates may be extended outside $S$; here $\lambda_j$ may not necessarily be in $[0,1]$ but the barycentric coordinates of  $z\in \CC^d$ may be found  by solving the linear system of equations (\ref{eqn:61}), 
(\ref{eqn:62}) for $\lambda_0,\ldots,\lambda_d$. 

\begin{remark} \rm 
Observe that $\lambda_0,\ldots,\lambda_d$ are linear in $z$, and when $z=p_j$, setting $\lambda_j=1$ and $\lambda_k=0$ ($k\neq j$)  solves the system (\ref{eqn:61}), (\ref{eqn:62}).  In addition, the polynomials of degree at most $1$ in $\CC^d$ form a space of dimension $d+1$.  It follows from this that as functions of $z$, $\lambda_j(z)$ must be the \emph{fundamental Lagrange interpolating polynomials} for $\{p_0,\ldots,p_d\}$, i.e., \def\VDM{\mathrm{VDM}}
$$
\lambda_j(z) = \frac{\VDM(p_0,\ldots,p_{j-1},z,p_{j+1},\ldots,p_d)}{\VDM(p_0,\ldots,p_d)},$$
 where for $\{b_0,\ldots,b_d\}\subset\CC^d$, 
 $$\VDM(b_0,\ldots,b_d) 
= \det\begin{bmatrix}
1& 1&\cdots & 1 \\
z_1(b_0)& z_1(b_1) & \cdots & z_1(b_d)\\
\vdots & \vdots & \ddots & \vdots\\
z_d(b_0)& z_d(b_1)& \cdots & z_d(b_d)
\end{bmatrix} 
$$
and $z_j(b_k)$ denotes the $j$-th coordinate of $b_k$.
\end{remark}

Let $h:\CC\setminus[-1,1]\to\CC\setminus\Delta$ denote the \emph{inverse Joukowski function},
$$
h(\eta) = \eta +\sqrt{\eta^2-1},
$$
where we take a branch of the square root that is postive on the positive real axis.  Recall that the \emph{Joukowski function}  is defined by $\zeta\mapsto \frac{1}{2}\bigl(\zeta+\tfrac{1}{\zeta}\bigr)$.  We have
\begin{equation}\label{eqn:hj}
h\left( \tfrac{1}{2}\bigl(\zeta + \tfrac{1}{\zeta}\bigr) \right)  = \zeta \hbox{ for all } \zeta\in\CC\setminus\Delta \end{equation}
and
\begin{equation}\label{eqn:hj2}
|h(\eta)| = h\left( \tfrac{1}{2}(|\eta+1| + |\eta-1|)\right).
\end{equation}
From the above equation, the level sets of $|h|$ are ellipses with foci at $\pm 1$.

\medskip

The following formula for $V_S$ was proved in \cite{bosmauwaldron:extremal}.

\begin{theorem}\label{thm:Vsimplex}
Given $z\in\CC^d$, let $\lambda_j=\lambda_j(z)$ (where $j=0,\ldots,d$) be the barycentric coordinates of $z$ as defined above.  Then
$$ V_{S}(z) = \log h(|\lambda_0(z)| + \cdots |\lambda_d(z)|) .$$\qed 
\end{theorem}

\begin{remark}\rm 
Let $u(z):= \log h(|\lambda_0(z)| + \cdots |\lambda_d(z)|)$.  By the properties of barycentric coordinates on $S$, we have $u(z)=0$ for all $z\in S$.  In \cite{bosmauwaldron:extremal} it is verified explicitly that  for all $z\in\CC^d\setminus S$, the matrix $\left[ \frac{\partial^2u}{\partial z_j\partial \bar z_k}(z) \right]_{j,k=1}^d$ is positive semidefinite of rank strictly less than $d$.  Thus  $u$ is a maximal psh function on $\CC^d\setminus S$, so the support of $(dd^cu)^d$ is contained in $S$.    A version of the global domination principle \cite{bedfordtaylor:plurisubharmonic} says that if $v$ is another  function in $\calL$ and $v\leq u$ on the support of $(dd^cu)^d$ (equivalently, where $u$ is not maximal), then $v\leq u$ in $\CC^d$.  Hence $u$ dominates any function in $\calL$ that is $\leq 0$ on $S$, i.e.,  any competitor for $V_S$.  So $u=V_S$.  
\end{remark}

\subsection*{Baran and Lundin formulas}

Let $\Sigma:=\co(\{\b0,\be_1,\ldots,\be_d\})$ denote the standard $d$-dimensional simplex in $\RR^d$, also given in terms of linear functions as
\begin{equation}\label{eqn:simplexdef}
\Sigma = \{x\in\RR^d\colon x_1\geq 0,\ldots,x_d\geq 0,\  x_1+\cdots+x_d\leq 1\}.
\end{equation}

  For $z=(z_1,\ldots,z_d)\in\CC^d$, we have 
$$
\lambda_1(z)=z_1,\ldots,\lambda_d(z)=z_d,  \hbox{ and } \lambda_0(z) = 1-z_1-z_2-\cdots - z_d, 
$$
which yields Baran's formula for the extremal function of the standard simplex,  
\begin{equation}\label{eqn:vs}
 V_{\Sigma}(z) = \log h(|z_1|+\cdots+|z_d|+|1-z_1-\cdots - z_d|) .
\end{equation}
  Using the square map $z=(z_1,\ldots,z_d)\mapsto (z_1^2,\ldots,z_d^2)=: Q(z)$ we obtain Lundin's formula for the extremal function of the real unit ball
$
B_{\RR} := B\cap\RR^d,
$   
where $B=\{z\in\CC^d: |z|\leq 1\}$ is the unit ball in $\CC^d$.  

\begin{corollary}[Lundin's formula]  \label{cor:lundin}
We have $$V_{B_{\RR}}(z) = \tfrac{1}{2}\log h(|z_1|^2+\cdots+|z_d|^2+|z_1^2+\cdots+z_d^2-1|).$$ 
\end{corollary}

\begin{proof}
From (\ref{eqn:simplexdef}) we obtain $$Q^{-1}(\Sigma) =\{x\in \RR^d\colon x_1^2\geq 0,\ldots,x_d^2\geq 0,\ x_1^2+\cdots + x_d^2\leq 1\},$$  which is clearly  the  real unit ball $B_{\RR}$.  Since $Q$ is a polynomial map of degree 2 with $Q^{-1}(0)=\{0\}$, 
\begin{equation}\label{eqn:VQ}
V_{B_{\RR}}(z) =  V_{Q^{-1}(\Sigma)}(z) = \tfrac{1}{2} V_{\Sigma}(Q(z)) \end{equation}
by Proposition \ref{prop:61}(2).  The result follows by plugging formula (\ref{eqn:vs}) into the right-hand side. \end{proof}

\subsection*{Robin functions}

Recall that the \emph{Robin function $\rho_K$} associated to a compact set $K\subset\CC^d$ is defined by 
$$
\rho_K(z)  = \limsup_{|\lambda|\to\infty}( V_K(\lambda z) - \log|\lambda|),
$$
and is logarithmically homogeneous: $\rho_K(\lambda z) = \rho_K(z) + \log|\lambda|$ for all $\lambda\in\CC^*=\CC\setminus\{0\}$.    
The \emph{Robin indicatrix} of $K$ is the set
$$
K_{\rho}:= \{z\in\CC^d\colon \rho_K(z)\leq 0\},
$$  
and is a polynomially convex set that is \emph{circled}: 
$$z\in K_{\rho}  \Longleftrightarrow e^{i\theta}z\in K_{\rho}, \quad\theta\in\RR.$$  
We will denote the Robin indicatrix of $B_{\RR}$ by $B_{\RR,\rho}$.

\begin{lemma} \label{lem:14b}
We have
$$
\rho_{B_{\RR}}(z) = \tfrac{1}{2} \log\left(2(|z_1|^2+\cdots+|z_d|^2+|z_1^2+\cdots+z_d^2|) \right).
$$
Therefore, $B_{\RR,\rho} = \{z\in\CC^d\colon |z_1|^2+\cdots+|z_d|^2+|z_1^2+\cdots+z_d^2|\leq\tfrac{1}{2}    \}$.
\end{lemma}

\begin{proof}
 For large values of  $|z|$, 
\[ \begin{aligned}
h(|z_1|^2+\cdots+|z_d|^2+ |z_1^2+\cdots z_d^2-1|) \ &= \  h(|z_1|^2+\cdots+|z_d^2|+|z_1^2+\cdots+z_d^2|) \\
& \hskip45mm + O(1)  \\
&= 2(|z_1|^2+\cdots+|z_d^2|+|z_1^2+\cdots+z_d^2|) \\
& \hskip45mm + O(1).
\end{aligned} \]
Hence 
\[ \begin{aligned}  V_{B_{\RR}}(z) &= \tfrac{1}{2}\log h(|z_1|^2+\cdots+|z_d|^2+ |z_1^2+\cdots z_d^2-1|) \\ &= \tfrac{1}{2} \log 2(|z_1|^2+\cdots+|z_d^2|+|z_1^2+\cdots+z_d^2|)    
+ O(\tfrac{1}{|z|}).  \end{aligned} \]
(Details of the `big-O' calculations are left to the reader.)  Finally,
$$
V_{B_{\RR}}(\lambda z) -\log|\lambda| = \tfrac{1}{2}\log 2(|z_1|^2+\cdots+|z_d^2|+|z_1^2+\cdots+z_d^2|)    
+ O(\tfrac{1}{|\lambda z|}).
$$
Letting $|\lambda|\to\infty$  yields the result. 
\end{proof}

Using (\ref{eqn:vs}), a similar calculation (omitted) shows that for the standard simplex $\Sigma$ in $\RR^d$, 
$$
\rho_{\Sigma}(z) = \log 2(|z_1|+\cdots+|z_d|+|z_1+\cdots+z_d|)$$
and therefore $$\Sigma_{\rho}= \{z\in\CC^d: |z_1|+\cdots+|z_d|+|z_1+\cdots+z_d|\leq\tfrac{1}{2} \}.$$

\section{Hooke ellipses}\label{sec:hooke}


\begin{definition}\rm
Let $K\subset\RR^d$ be a convex body and let $E\subset K$ be an ellipse.\footnote{The definition that follows will also extend to \emph{degenerate} ellipses; see Remark \ref{rem:95} below.}  Then $E$ is \emph{extremal for $K$} or  \emph{inscribed in $K$} if there is no translation that takes $E$ into the interior of $K$.
\end{definition}

\begin{remark}\label{rem:inscribed}   \rm
Inscribed ellipses are extremal in the following sense.  The largest ellipse contained in $K$ for a fixed eccentricity and orientation must be inscribed in $K$.  Otherwise, one could translate the ellipse into the interior of $K$ and expand it about the center to get a larger ellipse still contained in $K$ with the same eccentricity and orientation. 
\end{remark}

\begin{definition}\rm 
Let us call an ellipse $E$ in $\RR^d$ a \emph{Hooke ellipse} if it is centered at the origin.  In addition, if $E$ is inscribed in $B_{\RR}$ we will call it a \emph{Hooke ellipse for $B_{\RR}$}.  
 
More generally, let $K\subset\RR^d$ is a compact convex body that is \emph{symmetric (with respect to the origin)}, i.e., $x\in K\iff -x\in K$.  Then $E$ is a Hooke ellipse for $K$ if it is centered at the origin and inscribed in $K$.
\end{definition}

The complexification in $\CC^d$ of a real ellipse $E$ is denoted  $E_{\CC}$.   Using Lundin's formula, we can show that $V_{B_{\RR}}$ is harmonic 
on complexifications of Hooke ellipses for $B_{\RR}$.  Recall the standard notation $\CC^*=\CC\setminus\{0\}$ in what follows.

\begin{proposition}\label{prop:16}
 Let $E$ be a  Hooke ellipse for $B_{\RR}$.  There is a parametrization 
\begin{equation}\label{eqn:prop16}
\CC^* \ni \zeta \stackrel{f}{\longmapsto} c\zeta + \bar c/\zeta = z \in E_{\CC} 
\end{equation}
such that  $E=f(\partial\Delta)$,   $V_{B_{\RR}}(z) = V_{B_{\RR}}(f(\zeta))=   \bigl| \log|\zeta| \bigr|$, and $c\in \partial B_{\RR,\rho}$.  In particular, the restriction of $V_{B_{\RR}}$ to $E_{\CC}$ is harmonic on $E_{\CC}\setminus B_{\RR}$.  
\end{proposition}

\begin{proof}
For an explicit calculation, consider an ellipse $E$ whose major and minor axes are along the $\Re(z_1)$ and $\Re(z_2)$ coordinate axes respectively; then $E_{\CC}$ is given by the equations
$$
z_1^2+ \frac{z_2^2}{b^2}=1 \hbox{ where } b\in(0,1], \hbox{ and } z_j=0 \hbox{ for all } j>2.
$$
By Euler's formula, the trigonometric parametrization $z_1=\cos\theta$, $z_2=b\sin\theta$ of the real ellipse $E$ on the unit circle $\{\zeta = e^{i\theta}\colon \theta\in\RR\}$ yields the parametrization 
$$z_1 = \tfrac{1}{2}(\zeta+\tfrac{1}{\zeta}),\ z_2 = \tfrac{-ib}{2}(\zeta-\tfrac{1}{\zeta}),\quad \zeta\neq 0$$
of $E_{\CC}$.  In the  notation of (\ref{eqn:prop16}), $c=(\tfrac{1}{2}, -\tfrac{ib}{2},0,\ldots,0)$.  

We have
$|z_1|^2 = \tfrac{1}{4} | \zeta + \tfrac{1}{\zeta}|^2$ and $|z_2|^2 = \tfrac{b^2}{4}|\zeta - \tfrac{1}{\zeta} |^2$,  and  a computation gives 
$$
 z_1^2+z_2^2-1  =  \frac{1-b^2}{4}\left( \zeta-\frac{1}{\zeta}\right)^2,
$$
so that 
\[ \begin{aligned}
|z_1|^2 + |z_2|^2+|z_1^2+z_2^2-1|  &= \frac{1}{4}\left( \bigl|\zeta+\frac{1}{\zeta}\bigr|^2 + \bigl|\zeta - \frac{1}{\zeta}\bigr|^2 \right) \\
&= \frac{1}{2} \left( |\zeta|^2 + \frac{1}{|\zeta|^2}\right).
\end{aligned}\]
If $|\zeta|>1$ then 
\[\begin{aligned}
V_{B_{\RR}}(z) =  \log h(|z_1|^2 + |z_2|^2+|z_1^2+z_2^2-1|) & = \tfrac{1}{2} \log h\left( \tfrac{1}{2} \bigl( |\zeta|^2 + \tfrac{1}{|\zeta|^2}\bigr) \right) \\
&= \log|\zeta|,  \end{aligned}  \]
by (\ref{eqn:hj}).  Similarly, if $|\zeta|<1$ then $\tfrac{1}{|\zeta|}>1$ and $V_{B_{\RR}}(z) = \log \tfrac{1}{|\zeta|}$.  Putting the two formulas for $|\zeta|>1$ and $|\zeta|<1$ together,  and using the continuity of $V_{B_{\RR}}$ across $B_{\RR}$, we have 
$V_{B_{\RR}}(z) = \bigl|\log|\zeta|\bigr|$.  Thus the result holds for Hooke ellipses in the $(z_1,z_2)$-plane oriented along the coordinate axes. 

Observe that a real rotation $z\mapsto Rz$ is a linear map that leaves $B_{\RR}$ invariant, so $V_{B_{\RR}}(Rz)=V_{B_{\RR}}(z)$.  Let $E$ be a Hooke ellipse, and let   $R_{E}$ denote the real rotation  whose inverse takes $E$ to the Hooke ellipse oriented along the $z_1$ and $z_2$ axes with the same eccentricity.   A parametrization of $E_{\CC}$ is given by 
$$\zeta \stackrel{f}{\longmapsto} R_E\left( c_E\zeta + \bar c_E/\zeta \right)  \ = \ c\zeta + \bar c/\zeta $$
where $c_E=(\tfrac{1}{2},-\tfrac{ib}{2},0,\ldots,0)$ for some $b\in(0,1]$  and   $c=R_Ec_E$.  So we have a parametrization of the desired form, and  $V_{B_{\RR}}(f(\zeta))=\bigl|\log|\zeta|\bigr|$ by invariance. 

 It remains to check that $c\in \partial B_{\RR,\rho}$.  This follows from a calculation on $E_{\CC}$: for $|\zeta|>1$,  
\[ \begin{aligned}
0 = V_{B_{\RR}}(c\zeta + \bar c/\zeta) - \log|\zeta| &=  V_{B_{\RR}}(c\zeta + O(\tfrac{1}{|\zeta|})) - \log|\zeta| \\
&= V_{B_{\RR}}(c\zeta)+ O(\tfrac{1}{|\zeta|^2}) - \log|\zeta|, 
\end{aligned} \]
 where the  `big-O' calculations are similar to those done in Lemma \ref{lem:14b}.  Taking the limit as $|\zeta|\to\infty$, we obtain $\rho_{B_{\RR}}(c)=0$.  So $c\in\partial B_{\RR,\rho}$.  
\end{proof}

\begin{remark}\label{rem:95}  \rm
For any nonzero $c\in\CC^d$, parameters over the unit circle ($\zeta=e^{i\theta}$) plugged into  (\ref{eqn:prop16})  yield the real ellipse 
$$
E_c=\{2\Re(c)\cos\theta +2\Im(c)\sin\theta\colon \theta\in\RR\}.
$$
Hence when $c\in\RR^d$, we obtain the closed line segment joining $2c$ to $-2c$.  We consider this to be a \emph{degenerate} ellipse.  In this case, (\ref{eqn:prop16}) gives a parametrization of the Joukowski map in the complex line $L:=\{\lambda c\colon  \lambda\in\CC\}$.  The corresponding `complexified ellipse' is then this complex line $L$. The parametrization is a $2$-to-$1$ map from $\CC^*$ onto $L$.    

A real rotation sends $c$ to  $(\frac{1}{2},0,\ldots,0)$ with the image of the unit circle being the interval $[-1,1]$ in the $z_1$-plane.  The previous proposition holds for degenerate ellipses too, because the calculations in the proof remain valid when $b=0$.  In what follows, our arguments will apply to both degenerate and non-degenerate  ellipses unless they need to be treated separately.

The statement that a degenerate ellipse is inscribed in a convex set $K$ means that the endpoints of the line segment intersect $\partial K$.
\end{remark}

\begin{notation}\rm Given $c\in\partial B_{\RR,\rho}$, let us denote by $E_c$ and $E_{c,\CC}$ the real and complexified ellipses associated with the parametrization (\ref{eqn:prop16}). In general, we will use `$\CC$' as a subscript to denote complexification.
\end{notation}

It is easy to observe that the complex  span of $\{c,\bar c\}$ contains $E_{c,\CC}$ and the real span of $\{\Re(c),\Im(c)\}$ contains $E_c$.

\begin{proposition} \label{prop:18}
 Any $c\in \partial B_{\RR,\rho}$ defines a Hooke ellipse via the parametrization (\ref{eqn:prop16}).  The set $\partial B_{\RR,\rho}/\sim$ is a parameter space for Hooke ellipses for $B_{\RR}$, where $c'\sim c$ if $c'=\bar c$ or $c'= ce^{i\phi}$ for some $\phi\in\RR$.
\end{proposition}

\begin{proof}
Given $c\in\CC^d\setminus\{0\}$, let $E_c$ be the real ellipse parametrized by the unit circle, $e^{i\theta}\mapsto ce^{i\theta}+\bar ce^{-i\theta}$, and let$E_{c,\CC}$ be its complexification,  given by the parametrization (\ref{eqn:prop16}).  We claim that 
\begin{equation}\label{eqn:Ec}
c\in B_{\RR,\rho} \iff E_c\subset B_{\RR}.
\end{equation}
If $E_c\subset B_{\RR}$  then by the calculation in the previous proof, $u(\zeta):= V_{B_{\RR}}(c\zeta+\bar c/\zeta) -\log|\zeta|$ is a bounded subharmonic function on $\CC\setminus\Delta$ that goes to zero as $|\zeta|\to 1$ and goes to $\rho_{B_{\RR}}(c)$ as $|\zeta|\to\infty$.  Hence by the maximum principle (for the complement of the complex unit disk), $\rho_{B_{\RR}}(c)\leq 0$, so $c\in B_{\RR,\rho}$.    

Conversely, suppose $E_c$ contains points outside $B_{\RR}$.  Then $E_c$ is inscribed in a larger ball $\alpha B_{\RR}$ for some $\alpha>1$; rescaling, $E_{c/\alpha}$ is a Hooke ellipse.  Hence by the final calculation in the previous proof,
$$
0=\rho_{B_{\RR}}(c/\alpha) = \rho_{B_{\RR}}(c) -\log|\alpha|
$$
so $\rho_{B_{\RR}}(c)=\log|\alpha|>0$, i.e., $c\not\in B_{\RR,\rho}$.  This proves the claim.

We can see that the Hooke ellipses are given by $\partial B_{\RR,\rho}$ as follows.  If $c\in\intt(B_{\RR,\rho})$ then $\alpha c\in B_{\RR,\rho}$ also, so $E_{\alpha c}\subset B_{\RR}$ for $\alpha>1$ sufficiently close to 1.  One can make a sufficiently small translation of $E_{c}$ along a direction in the plane spanned by its axes, so that the translated ellipse remains in the planar region bounded by $E_{\alpha c}$.  Hence $E_c$ is not inscribed in $B_{\RR}$, i.e., is not a Hooke ellipse.

If $c'=e^{i\phi}c$ then 
$$
E_{c'} = \{ c'e^{i\theta}+\overline{c'}e^{-i\theta}\colon \theta\in\RR\} = \{ce^{i(\phi+\theta)} + \bar{c}e^{-i(\phi+\theta)}\colon \theta\in\RR\} = E_c.
$$
 Another calculation shows that $c'=\bar c$ also implies $E_{c'}=E_c$.  So $c'\sim c$ implies $E_{c'}=E_c$ and the last statement follows.
\end{proof}

One can generalize the previous propositions to a solid (i.e., filled-in) ellipsoid.  Note that any solid ellipsoid $A$ centered at the origin is the image of $B_{\RR}$ under a nonsingular linear map.  As with the ball, a Hooke ellipse for $A$ is defined to be an ellipse inscribed in $A$.  A straightforward application of Proposition \ref{prop:61}(2) yields the following.

\begin{corollary}
Let $A$ be a solid ellipsoid centered at the origin, and let $L$ be the nonsingular linear map for which $A=L(B_{\RR})$.  Any Hooke ellipse for $A$ must be an ellipse $E_c$ whose complexification $E_{c,\CC}$ has a  parametrization as in (\ref{eqn:prop16}) with $E_c=f(\partial\Delta)$, $V_{A}(z)=V_{A}(f(\zeta))=\left|\log|\zeta|\right|$, and $c\in\partial A_{\rho}$.  Moreover, $A_{\rho}=L(B_{\RR,\rho})$ and the ellipse $E_{\tilde c}$ (where $\tilde c=L^{-1}(c)$)  is a Hooke ellipse for $B_{\RR}$.  \qed
\end{corollary}

Hooke ellipses satisfy the following `identity principle'.

\begin{theorem}\label{prop:19}
Let $K_1,K_2$ be compact sets, and let $E_1\subset K_1,E_2\subset K_2$ be Hooke ellipses whose complexifications $E_{1,\CC}, E_{2,\CC}$ are parametrized  as in (\ref{eqn:prop16}) by maps $f_1,f_2$.  Suppose  
$$V_{K_1}(f_1(\zeta))=\log|\zeta|,\ V_{K_2}(f_2(\zeta))=\log|\zeta|,\quad \hbox{for all } \zeta\in\CC\setminus\Delta. $$
Further, suppose $z_0\in E_{1,\CC}\cap E_{2,\CC}$ and $V_{K_1}(z_0)=V_{K_2}(z_0)>0$.  Then $E_{1,\CC}=E_{2,\CC}$.  
\end{theorem}

In the proof of the theorem we will need to relate the extremal function on $\CC^d$ to the extremal function on a subspace.  We state a result that follows from results of Sadullaev (\cite{sadullaev:estimate}; see also Theorem 2.2 of \cite{hartmau:chebyshev}).  If $K\subset A\subset\CC^d$ where $A$ is an algebraic variety and $K$ is compact, write  $V_{A,K}$ for the extremal function given \emph{intrinsically} on the manifold of regular points of $A$, defined as the upper envelope of plurisubharmonic functions on the manifold of logarithmic growth at infinity.
\begin{theorem}\label{thm:sadullaev}
Let $A\subset\CC^d$ be an algebraic variety of pure dimension $<d$ and let $K\subset A$ be a compact subset such that $V_{A,K}(z)<\infty$.   Then $V_{K}(z)=V_{A,K}(z)$ for all $z\in A$. \qed
\end{theorem}

\begin{remark}\rm
Recall $V_K$ is defined on all of $\CC^d$.  From general pluripotential theory and the fact that $A$ is pluripolar, $V_K^*(z)\equiv+\infty$  (\cite{klimek:pluripotential}, Chapter 5).    \end{remark}

\begin{proof}[Proof of Theorem \ref{prop:19}]
First, $z_0=f_1(\zeta_1)=f_2(\zeta_2)$ for some $\zeta_1,\zeta_2\in\CC\setminus\Delta$, and  
\begin{equation}\label{eqn:VAj}
\log|\zeta_1|= V_{K_1}(f(\zeta_1))=  V_{K_1}(z_0)= V_{K_2}(z_0)=    V_{K_2}(f(\zeta_2)) = \log|\zeta_2|.\end{equation}

If $E_1,  E_2$ are degenerate then the complexifications $E_{1,\CC},  E_{2,\CC}$ are complex lines through the points $0$ and $z_0$.  Hence they must be the same line, i.e., $E_{1,\CC}=E_{2,\CC}$.  For $E_{1,\CC}\neq E_{2,\CC}$ to hold, at least one ellipse must be non-degenerate.  To rule this out, assume $E_1$ is non-degenerate, and denote by $P_1$ the 2-dimensional plane containing $E_1$.  Its complexification $P_{1,\CC}$ then contains $E_{1,\CC}$.  We consider separately the cases in which the real intersections are empty or nonempty.

\begin{description}
\item[{Case 1:} $E_1\cap E_2=\emptyset$.]\ \\
Suppose the other ellipse $E_2$ is degenerate.  Then $E_{2,\CC}$ is the complex line passing through the points $\{0,z\}$, which are contained in $P_{1,\CC}$.   Hence $E_{2,\CC}\subset P_{1,\CC}$.  Looking at the real points,  it follows that $E_2$ is a line segment contained in $P_1$.  

Let  $A:=P_{1,\CC}$ in what follows.   Clearly $E_2$ must be contained in the interior of $U:=\text{co}(E_1)$ since it cannot intersect the boundary $\partial U= E_1$.  Then $E_2$ is a Hooke ellipse for the set $U_{\delta}:=(1-\delta)U$, for some $\delta>0$ chosen to shrink $U$ so that $E_2$ meets $\partial U_{\delta}$.  Then 
$$
\log|\zeta_2|= V_{A,U_{\delta}}(f(\zeta_2)) = V_{A,U_{\delta}}(z_0).
$$
On the other hand, 
$$V_{A,U_{\delta}}(z_0)  >V_{A,U}(z_0)\geq V_{A,K_1\cap A}(z_0)= V_{K_1\cap A}(z_0)\geq V_{K_1}(z_0)= \log|\zeta_1|,$$
where we use Theorem \ref{thm:sadullaev} to equate the third and fourth terms. We get a  contradiction to (\ref{eqn:VAj}).  

Suppose now that $E_2$ is non-degenerate and lies in the plane $P_1$.  Then since both ellipses are  symmetric about the origin in that plane and do not intersect, one ellipse must be contained in the interior of the convex hull of the other; say $E_2$ is in the interior of $\text{co}(E_1)$, as above.  Applying the same argument as in the previous paragraph yields a contradiction.

Finally, suppose $E_2$ is nondegenerate and lies in a plane $P_2$ with $P_1\neq P_2$. Then $P_1\cap P_2=:L$ is a line through the origin.  Also, 
$$E_{1,\CC}\cap E_{2,\CC}\subset P_{1,\CC}\cap P_{2,\CC}= L_{\CC} \ \Longrightarrow \  E_{1,\CC}\cap E_{2,\CC}\subseteq E_{1,\CC}\cap L_{\CC}.$$ 
 By B\'ezout's Theorem, $E_{1,\CC}\cap L_{\CC}$ consists of at most two points.  By elementary geometry in the plane, the intersection $E_1\cap L$  (of a line through the origin with an ellipse centered at the origin) is exactly two real points.   Therefore these points are the entire intersection $E_{1,\CC}\cap L_{\CC}$.  However, $z_0\in E_{1,\CC}\cap E_{2,\CC}$, and is not real, which gives a contradiction. 

\item[{Case 2:} $E_1\cap E_2\neq \emptyset$.]\ \\
 Let $a\in E_1\cap E_2$.  By symmetry about the origin, $-a\in E_1\cap E_2$ also.  

If $E_2$ is degenerate then $E_{2,\CC}$ is a complex line and 
$E_{1,\CC}\cap E_{2,\CC}\supseteq \{a, -a, z_0\}$ which is 3 points. By B\'ezout's Theorem, $E_{1,\CC}\cap E_{2,\CC}$ can be at most 2 points (counting multiplicity), so we have a contradiction.

Hence $E_2$ must also be non-degenerate, and contained in a 2-dimensional plane $P_2$. If $P_2\neq P_1$ then $P_{2,\CC}\cap P_{1,\CC}$ is a complex line $L_{\CC}$ containing $E_{1,\CC}\cap E_{2,\CC}$, and in particular, the points $\{a,-a,z_0\}$.  But then these 3 points are contained in the set $L_{\CC}\cap E_{1,\CC}$ which can have at most 2 points, and again we get a contradiction.

So $E_2$ is non-degenerate and lies in $P_1$. We restrict to this plane in what follows. 

If $a\in E_1\cap E_2$ is an intersection point of multiplicity 2, then by symmetry,  $-a$ is also an intersection point of multiplicity 2.  Then the intersection $E_{1,\CC}\cap E_{2,\CC}$ has multiplicity at least $5$---twice at each of $\pm a$ and at least once at $z_0$.  But by B\'ezout's theorem, two complex ellipses cannot intersect with multiplicity greater than 4 unless they are equal.  So $E_{1,\CC}=E_{2,\CC}$.

If $a\in E_1\cap E_2$ is an intersection point of multiplicity 1, then the curves meet transversally at $a$.  Fixing an orientation (or parametrization) of $E_1$, we may assume that $E_1$ \emph{enters} the convex hull  $\text{co}(E_2)$, i.e., the filled-in ellipse bounded by $E_2$, at $a$.  By symmetry, $E_1$ also enters $\text{co}(E_2)$ at $-a$.  Hence along $E_1$, between $a$ and $-a$, there is a point $b\in E_1\cap E_2$ where $E_1$ \emph{exits} $\text{co}(E_2)$.  By symmetry, $-b\in E_1\cap E_2$ also.  Hence $\{a,-a,b,-b,z_0\}\subset E_{1,\CC}\cap E_{2,\CC}$.  Again by B\'ezout's theorem, $E_{1,\CC}=E_{2,\CC}$.  
\end{description}
Altogether, only Case 2 can occur, and in that case $E_{1,\CC}=E_{2,\CC}$.
\end{proof}

Take $K_1=K_2=B_{\RR}$.  The proposition shows that for parameters $c'\not\sim c$ (i.e., different Hooke ellipses), the complexifications are disjoint outside the real ball: $(E_{c',\CC}\cap E_{c,\CC})\setminus B_{\RR} =\emptyset$.

\begin{lemma}\label{lem:110}
The complexifications of Hooke ellipses for $B_{\RR}$ give a continuous foliation of a subset of $\CC^d\setminus B_{\RR}$.  
\end{lemma}

\begin{proof} 
By the previous result, $z_0\in E_{c,\CC}\cap E_{c',\CC}$ implies $E_{c,\CC} =  E_{c',\CC}$.  So $E_{c,\CC}$ and $E_{c',\CC}$ are  disjoint outside $B_{\RR}$ when $c'\not\sim c$.   
Also, if $c'\to c$ then $f_{c'}\to f_c$   locally uniformly for $\zeta\in\CC^*$, since the parametrization (\ref{eqn:prop16}) is rational (and hence holomorphic) in 
$\zeta$.  So the sets $E_{c,\CC}$ vary continuously in $c$, i.e., we get a continuous foliation. 
\end{proof}

In fact, the foliation fills the entire set $\CC^d\setminus B_{\RR}$, which will be a consequence of the next result.  Define the \emph{(generalized) Joukowski map}
$\calR \colon \CC^d\setminus B_{\RR,\rho} \to\CC^d\setminus B_{\RR}$ 
by the formula 
\begin{equation}\label{eqn:rem}
\calR(c\zeta) := c\zeta +\bar c/\zeta,\quad c\in\partial B_{\RR,\rho},\ |\zeta|>1.
\end{equation}

\begin{proposition}\label{prop:jhomeo}
The Joukowski map is well-defined, and is a homeomorphism. 
\end{proposition}

\begin{proof}
We first verify that $\calR$ is a well-defined function on its domain.  If $z\in\CC^d\setminus B_{\RR,\rho}$ then the line segment $\{tz\colon t\in[0,1]\}$ joining the origin to $z$ intersects $\partial B_{\RR,\rho}$ in a point $c=\tau z$ for some $\tau\in(0,1)$; thus applying formula (\ref{eqn:rem}) to $c/\tau$ gives a value for $\calR(z)$.  We need to show that the formula for $\calR(z)$ gives the same value for any other $c'\in\partial B_{\RR,\rho}$ and $|\eta|>1$ such that $z=c'\eta$.  By Propositions \ref{prop:18} and  \ref{prop:19}, the point $z$ can be on at most one complexified ellipse, and this occurs for parameters in which $c'\sim c$.  If $c'=\bar c$ and $z=c\zeta=\bar c\eta$ then 
$\zeta=\frac{\bar c}{c}\eta$ and
$$
\calR(c\zeta) = c\zeta + \bar c/\zeta = c\tfrac{\bar c}{c}\eta + {\bar c}/ \left(\tfrac{\bar c}{c}\eta \right) = \bar c\eta + c/\eta = c'\eta + \overline{c'}/\eta = \calR(c'\eta).
$$
If $c'=ce^{i\theta}$ for some $\theta\in\RR$ then a similar calculation also gives $\calR(c\zeta)=\calR(c'\eta)$.  Thus we have a well-defined function.  From the formula, it is easy to verify that $\calR$ is continuous.

To show that $\calR$ is one-to-one, suppose $\calR(c\zeta)=\calR(c'\eta)$.  Then 
$$
E_{c,\CC} \ni c\zeta+\bar c/\zeta = c'\eta + \overline{c'}/\eta \in E_{c',\CC},
$$
i.e., $(E_{c,\CC}\cap E_{c',\CC})\setminus B_{\RR} \neq\emptyset$.  By Propositions \ref{prop:19} and \ref{prop:18},  $E_{c,\CC}=E_{c',\CC}$ and $c\sim c'$.  When $c'=\bar c$ we have
$$
c\zeta + \bar c/\zeta = \bar c\eta + c/\eta \  \Longleftrightarrow \  \eta\zeta(c\zeta-\bar c\eta) = c\zeta-\bar c\eta.
$$
Since $|\eta|,|\zeta|>1$ we must have $c\zeta=\bar c\eta$. 

 When $c'=ce^{i\phi}$ a similar calculation yields
$$
e^{i\phi}\eta\zeta c(\zeta - e^{i\phi}\eta) \ = \ \bar c(\zeta - e^{i\phi}\eta) \  \Longrightarrow \  |\eta\zeta| |\zeta - e^{i\phi}\eta| = |\zeta - e^{i\phi}\eta|.
$$
Since $|\eta|,|\zeta|>1$ we must have $\zeta = e^{i\phi}\eta$.  In both cases, we obtain $c\zeta=c'\eta$, proving that $\calR$ is one-to-one.

Since $\calR$ is continuous and injective, it is a homeomorphism onto its image by the domain invariance theorem in topology.  It remains to show that this image is all of $\CC^d\setminus B_{\RR}$.  Suppose for a contradiction that $\Omega:=\calR(\CC^d\setminus B_{\RR,\rho})$ is a proper open subset of $\CC^d\setminus B_{\RR}$.  Then there exists a $w\in\partial\Omega\setminus B_{\RR}$, and $V_{B_{\RR}}(w)=\epsilon$ for some $\epsilon>0$.  Form a sequence $w_n$ of points in $\Omega$ with $w_n\to w$.  We have $w_n=\calR(c_n\zeta_n)$ where $c_n\in\partial B_{\RR,\rho}$ and $|\zeta_n|>1$, and the ellipses  $E_{c_n}$ are inscribed in $B_{\RR}$ for each $n$.  By a standard compactness argument, and passing to a subsequence, we may assume that $c_n\to c\in\partial B_{\RR,\rho}$ and $\zeta_n\to\zeta$.  Since $V_{B_{\RR}}(w_n)=\log|\zeta_n|$ we have $V_{B_{\RR}}(w)=\log|\zeta|$ by continuity, so $|\zeta|=e^{\epsilon}>1$.  Hence $w=\calR(c\zeta)\in\Omega$, contradicting that $w\in\partial\Omega$.
\end{proof}

\begin{corollary}
The complexifications of Hooke ellipses give a continuous foliation of $\CC^d\setminus B_{\RR}$ such that $V_{B_{\RR}}$ is harmonic on each leaf of the foliation.  
\end{corollary}

\begin{proof}
The Joukowski map is onto, so for any $z\in\CC^d\setminus B_{\RR}$ we have $z=\calR(c\zeta) = c\zeta+\bar c/\zeta$ for some $c\in\partial B_{\RR,\rho}$ and $|\zeta|>1$.  Hence $z\in E_{c,\CC}$, i.e., $z$ is an element of the subset of $\CC^d\setminus B_{\RR}$ foliated by complexifications of Hooke ellipses (see Lemma \ref{lem:110}).  Since $z$ was arbitrary, this set must be all of $\CC^d\setminus B_{\RR}$.
\end{proof}

Theorem \ref{prop:19} also yields a formula for the extremal function of the intersection of two solid ellipsoids about the origin.  A generalization will be proved in Theorem \ref{thm:symmetric}.  

\begin{proposition}\label{thm:ellipsoid}
Let $A_1,A_2$ be solid ellipsoids in $\RR^d$ centered at the origin.  Let $A:=A_1\cap A_2$.  Then 
$$
V_A(z)=\max\left\{ V_{A_1}(z),  V_{A_2}(z)   \right\}.  
$$
\end{proposition}

\begin{proof}
Denote the right-hand side by $W(z)$.  Then $W$ is plurisubharmonic and $W(z)=0$ for all $z\in A$.  Hence $W\leq V_A$.  It remains to prove the reverse inequality at points outside $A$.  

 Let $z_0\in\CC^d\setminus A$.  If $W(z_0)>V_{A_1}(z_0)$ then by continuity this inequality holds in an open set, i.e., $W(z)>V_{A_1}(z)$ for all $z$ in a neighborhood of $z_0$.  Hence $W(z)=V_{A_2}(z)$ in that neighborhood, so $W$ is maximal there.  

Similarly, if $W(z_0)>V_{A_2}(z_0)$ then $W$ is maximal in a neighborhood of $z_0$.  

Otherwise $W(z_0)=V_{A_1}(z_0)= V_{A_{2}}(z_0)$.  Now $z_0$ lies on the complexifications  of Hooke ellipses $E_1,E_2$ for $A_1,A_2$ respectively. By Theorem \ref{prop:19}, $E_{1,\CC}=E_{2,\CC}$.  Hence $E_1=E_2$ which is an ellipse $E$ inscribed in both $A_1$ and $A_2$.  So $z_0\in E_{\CC}$ and $E$ is inscribed in $A_1\cap A_2=A$. 

Let $f$ denote the parametrization of $E_{\CC}$ such that $V_{A_1}(f(\zeta))=\log|\zeta|$ for $|\zeta|\geq 1$.  Note that $V_{A_2}(f(\zeta))=\log|\zeta|$ also.  So $W(f(\zeta))=\log|\zeta|$.   Given $u\in\calL$ such that $u\leq 0$ on $A$, the function 
\begin{equation*}
\varphi(\zeta):=u(f(\zeta))-W(f(\zeta))=u(f(\zeta))-\log|\zeta|\end{equation*}
is subharmonic on $\CC\setminus\Delta$ and $\varphi\leq 0$ on $\partial\Delta$.  Hence $\varphi\leq 0$ on $\CC\setminus\Delta$ by the maximum principle.  In particular, $u(z_0)\leq W(z_0)$.  Taking the sup over all such $u$,  $V_A(z_0)\leq W(z_0)$.  

Altogether, for any $z\in\CC^d$, $W$ is either maximal in a neighborhood of $z$ or $V_A(z)\leq W(z)$.  By the domination principle, $V_A\leq W$ and therefore $V_A=W$.  
\end{proof}

\section{Newton ellipses} \label{sec:newton}

Let $c\in\partial B_{\RR,\rho}$ and suppose $z=c\zeta +\bar c/\zeta$ where $|\zeta|\geq 1$.  Then for $j=1,\ldots,d$, 
\begin{equation}\label{eqn:newtonparam}
z_j^2 = (c_j\zeta + \bar c_j/\zeta)^2  =   c_j^2\zeta^2 + \bar c_j^2/\zeta^2 + 2|c_j|^2. 
\end{equation}
When $\zeta=e^{i\theta}$ the right-hand side parametrizes a real ellipse in $\RR^d$ centered at $(2|c_1|^2,\ldots,2|c_d|^2)$; the plane of the ellipse is the translation to this point of the  subspace spanned by $\{\Re(c^2),\Im(c^2)\}$, where we write $c^2:=(c_1^2,\ldots, c_d^2)$.

\begin{definition}\label{def:91}   \rm 
A \emph{Newton ellipse} is the image of a Hooke ellipse  under the square map $Q(z_1,\ldots,z_d)=(z_1^2,\ldots,z_d^2)$.    
A \emph{Newton ellipse for $\Sigma=Q(B_{\RR})$} is the image of  a Hooke ellipse for $B_{\RR}$. 

More generally, let $\lambda\in(0,\infty)$.  Then $Q(\lambda z) = \lambda^2Q(z)$, and it follows that  $Q(\lambda B_{\RR}) = \lambda^2\Sigma$.  Given a Hooke ellipse $E$ for $\lambda B_{\RR}$, its image $Q(E)\subset\lambda^2\Sigma$ is an ellipse contained in $\lambda^2\Sigma$ .  We will define this to be a \emph{Newton ellipse for $\lambda^2\Sigma$}.
\end{definition}

\begin{remark}\rm
Classically, a Hooke ellipse in the complex plane is an ellipse with foci at $-1,1$, while a Newton ellipse has foci at $0,1$.  A Newton ellipse is the image of a Hooke ellipse under the square map $z\mapsto z^2$  (see e.g., \cite{trefethen:multivariate}).   

In our higher-dimensional setting, let $E_{\CC}$ and $Q(E_{\CC})$ be complexified Hooke and Newton ellipses as given in Definition \ref{def:91}.  If we fix $r>1$, then $\zeta\mapsto c\zeta+\bar c/\zeta$ restricted to $|\zeta|=r$ is  a real 1-dimensional ellipse $E(r)$ whose image $Q(E(r))$ is another real 1-dimensional ellipse with a parametrization given by the right-hand side of (\ref{eqn:newtonparam}).  This is an exact analogue of the classical relation.  Note that $E_{\CC}=\bigcup_{r\in(0,\infty)} E(r)$ and $Q(E_{\CC})=\bigcup_{r\in(0,\infty)} Q(E(r))$. 
\end{remark}


Using the calculation (\ref{eqn:newtonparam}), it is easy to see the following.

\begin{lemma} \label{lem:114}
Suppose $a\in\RR^d$ and $c\in\CC^d$ parametrize the complexification of an ellipse $E$ via
$$
\zeta \mapsto a + c\zeta + \bar c/\zeta, \quad \zeta\in\CC^*.
$$
Then $E$ is a Newton ellipse if and only if $a_j=2|c_j|$ for all $j=1,\ldots, d$.  \qed
\end{lemma}

The aim in what follows is to characterize Newton ellipses geometrically as inscribed ellipses.  

\begin{lemma}\label{lem:118}
An ellipse $E\subset \Sigma$ is inscribed in  $\Sigma$ if and only if it intersects every codimension 1 face.
\end{lemma}

\begin{proof}
Denote the faces of $\Sigma$ by $F_0,\ldots,F_d$, which lie on the hyperplanes $H_0,\ldots,H_d$ respectively.  To prove one implication, suppose $E$ does not intersect, say, $F_0$.  By convexity, $E$ is contained in a cone bounded by $F_1,\ldots, F_d$ and their point of intersection $p$.  Consider a translation $v+E$ into the interior of the cone for some small vector $v$.  Then by construction, $v+E$ does not intersect $F_1,\ldots,F_d$, and if $|v|$ is sufficiently small, $v+E$ will not intersect $F_0$ either.  It follows that $v+E$ is interior to $\Sigma$, so $E$ is not inscribed in $\Sigma$.

On the other hand, if $E$ is not inscribed in $\Sigma$, then $v+E$ is interior to $\Sigma$ for some vector $v$.  Now $v\cdot n_j<0$ for some inward normal vector $n_j$ to a face $F_j$.  In this case, translation by $v$ takes points in the interior of $\Sigma$ closer to $H_j$.  But $E=-v+ (v+E)$, so points of $E$ are translations of points of $v+E$ away from $H_j$.  Since $v+E$ does not intersect $H_j$, neither does $E$.  So $E$ does not intersect some face of $\Sigma$.
\end{proof}

\begin{lemma}\label{lem:uniqueellipse}
For a given eccentricity and orientation in $\RR^d$, there is a unique ellipse $E$ inscribed in $\Sigma$.
\end{lemma}

\begin{proof}
Let $E,\tilde E$ be inscribed in $\Sigma$ with the same eccentricity and orientation.      Let $F_0,\ldots,F_d$ be the faces and $n_0,\ldots,n_d$ the corresponding normal vectors.  Since they have the same eccentricity and orientation, and must be of the same size, $\tilde E$ is a translation of $E$ by some vector $v$.  Using convexity of $\Sigma$, it is easy to see that $\epsilon v+ E$ must also be inscribed in $\Sigma$ for every $\epsilon\in[0,1]$.  

Suppose $v$ translates a point of $E\cap F_j$ into the interior of $\Sigma$, for some $j$.      Then $v\cdot n_j>0$ and by Lemma  \ref{lem:119}, $v\cdot n_k<0$ for some $k\neq j$.\footnote{To apply the lemma precisely, we also need the fact that some multiple of $v$ translates a vertex of $F_j$ into the interior of $\Sigma$; this is easy to see.}  By Lemma \ref{lem:118}, $E\cap F_k$ is nonempty, and $v$ translates any point of $E\cap F_k$ outside $\Sigma$.  Hence $\tilde E$ contains points outside $\Sigma$, a contradiction.  So $v$, and by the same reasoning, any multiple $\epsilon v$ (where $\epsilon\in[0,1]$), cannot translate points of $\partial\Sigma$ into the interior of $\Sigma$. 

In other words, if $x$ is a point of $\partial\Sigma$ then $x+\epsilon v\in\partial\Sigma$ for any $\epsilon\in[0,1]$.  Hence if $x$ is contained in some codimension 1 face, then the line segment joining $x$ to $x+v$ must stay in that face.  By Lemma \ref{lem:118}, both $E$ and $\tilde E$ meet every face of $\Sigma$.  For each $j=0,\ldots,d$ let $x_j\in F_j\cap\partial E$;  then the line segment joining $x_j$ to $x_j+v$ is contained in $F_j$.  This implies that $n_j\cdot v=0$ for each $j$.  Since the normal vectors $n_0,\ldots,n_d$ span $\RR^d$, we have $v=0$.  
\end{proof}

We next restrict to the $2$-dimensional case.  To emphasize this, we write $\Sigma_2$ for the standard simplex  in $\RR^2$ (i.e.,  the triangle $\text{co}\{(0,0),(1,0),(0,1)\}$).  
By the previous result, $E\subset\Sigma_2$ is an inscribed ellipse if and only if $E$ intersects each edge of $\Sigma_2$.  We will also restrict our attention to non-degenerate ellipses.




\begin{lemma} \label{lem:105}
Let $E,\tilde E$ be non-degenerate ellipses inscribed in $\Sigma_2$.  Suppose they intersect two edges of $\Sigma_2$ in the same points.  Then $E=\tilde E$.
\end{lemma}

\begin{proof}
Suppose $\{s_j\} = E\cap f_j = \tilde E\cap f_j$ for $j=1,2$, say.  The ellipse $E$ is given by an implicit equation $$x_1^2+a_{02}x_2^2+ a_{11}x_1x_2 +a_{10}x_1+a_{01}x_2+a_{00}=0$$
The fact that $E$ is inscribed means that it intersects $s_1,s_2$ tangent to the  edges $f_1,f_2$; this yields 4  linear equations in the parameters $a_{20},a_{11},a_{10},a_{01},a_{00}$.  It also intersects $f_3$ tangentially, which yields one more linear equation in the parameters.  Hence we obtain a system of $5$ linear equations in $5$ unknowns, so its solution is unique.  By hypothesis, we obtain the same system of equations for the parameters of both $E$ and $\tilde E$.  So $E=\tilde E$.  
\end{proof}

\begin{proposition}\label{prop:117}
Let $\lambda>0$.   Any non-degenerate Newton ellipse for $\lambda^2\Sigma_2$ is inscribed in $\lambda^2\Sigma_2$; and conversely, any non-degenerate ellipse inscribed in $\lambda^2\Sigma_2$ is a Newton ellipse for $\lambda^2\Sigma_2$. 
\end{proposition}

\begin{proof}
We consider when $\lambda=1$.  If $E$ is a Newton ellipse then it is contained $\Sigma_2$.  By definition, it is the image $E=Q(\tilde E)$ of a Hooke ellipse $\tilde E$ for the real unit disk 
$B_{\RR}$.  Let $(c,d)$ be a point in the intersection  $\tilde E\cap \partial B_{\RR}$.  Since $\tilde E$ is symmetric about the origin ($x\in \tilde E\iff -x\in\tilde E$), it contains points whose coordinates contain both positive and negative values.  By continuity, there are points of $\tilde E$ whose coordinates are zero; let  $(a,0),(0,b)$ denote points on each axis. Then $E=Q(\tilde E)$  intersects the boundary of $\Sigma_2$ in the points $(a^2,0),(0,b^2),(c^2,d^2)$ which are on each face of $\Sigma_2$.  Hence by Lemma \ref{lem:118},  $E$ is inscribed in $\Sigma_2$.   

Conversely, suppose $E$ is an ellipse inscribed in $\Sigma_2$.  Then it must intersect the vertical and horizontal edges tangentially; say at $(a,0),(0,b)$, with $a,b\in(0,1)$.   We need to show that $E$ is a Hooke ellipse.  To this end,  construct an ellipse $\tilde E$ centered at the origin that goes through the points $(\sqrt{a},0),(0,\sqrt{b})$  and meets $x_1^2+x_2^2=1$ tangentially in the first quadrant, say in $(\tilde c,\tilde d)$, where $\tilde c,\tilde d\in(0,1)$.  By symmetry of ellipses about the origin, $\tilde E$ also goes through the points $(-\sqrt{a},0)$, $(0,-\sqrt{b})$, and $(-\tilde c,-\tilde d)$.  Hence $\tilde E$ is inscribed in $B_{\RR}$, so is a Hooke ellipse.  The image $Q(\tilde E)$ is then a Newton ellipse that is inscribed in $\Sigma_2$, going tangentially  through the points $(a,0),(0,b),(\tilde c^2,\tilde d^2)$.  Since $E$ is inscribed in $\Sigma_2$ and also goes through $(a,0),(0,b)$ tangentially, we have $Q(\tilde E)=E$ by the previous lemma, i.e., $E$ is a Newton ellipse.

When $\lambda\neq 1$, apply the above proof after rescaling by  $(z_1,z_2)\mapsto (z_1/\lambda,z_2/\lambda)$; this transforms $\lambda^2 \Sigma_2$ to $\Sigma_2$ and $\lambda B_{\RR}$ to $B_{\RR}$.
\end{proof}


We prove the general case $d\geq 2$ (with $\Sigma=Q(B_{\RR})\subset\RR^d$) by considering projections to 2 dimensions.

\begin{proposition}\label{prop:97}
A non-degenerate ellipse $E$ is a Newton ellipse for $\Sigma$  if and only if it is inscribed in  $\Sigma$.
\end{proposition}

\begin{proof}
We first set up some notation.  Let $H_0$ be the hyperplane given by $z_1+\cdots +z_d=1$, and let $H_j$ be the hyperplane given by $z_j=0$ ($j=1,\ldots,d$).  These are the  hyperplanes containing each codimension 1 face of $\Sigma$.  Given a non-degenerate ellipse $E$, let the parametrization of its complexification be given by
\begin{equation}\label{eqn:prop119}
\CC^*\ni \zeta\mapsto a + c\zeta + \bar c/\zeta \in \CC^d, \quad (a \in\RR^d, c \in\CC^d).
\end{equation}
Here $a$ is the center of the ellipse, and $P_E:=\span\{\Re(c),\Im(c)\}\subset\RR^d$ is the real plane containing the ellipse.

Suppose $E$ is a Newton ellipse.  Then the same argument as in the first part of the proof of Proposition \ref{prop:117} works here.  Write $E=Q(\tilde E)$, where $\tilde E$ is a Hooke ellipse for $B_{\RR}$.  Using the symmetry of $\tilde E$ about the origin and the fact that it intersects $\partial B_{\RR}$, we obtain, upon mapping forward to $\Sigma$, that $E$ intersects each hyperplane $H_j$.  Hence by Lemma \ref{lem:118}, $E$ is  inscribed in $\Sigma$.

Now suppose $E$ is a non-degenerate ellipse that is  inscribed in $\Sigma$.      Choose a projection $\pi$ to two of the coordinates such that, restricted to $P_E$, $\pi$ is a bijection.   In what follows, we suppose without loss of generality that $\pi(z_1,\ldots,z_d)=(z_j,z_k)$.  
The projected ellipse $\pi(E)$ has parametrization
$$
\zeta\mapsto (a_j + c_j\zeta + \bar c_j/\zeta, a_k + c_k\zeta+\bar c_k/\zeta).
$$
From this, the condition that $\pi$ is a bijection says that 
\begin{equation}\label{eqn:nonsing}
 M:=\begin{bmatrix} c_j & \bar c_j \\ c_k & \bar c_k  \end{bmatrix} \hbox{ is nonsingular, i.e., } \det(M)=2i\Im(c_j\bar c_k)\neq 0.\end{equation}
  This holds whenever $c_j,c_k$ are not real multiples of each other.  

Suppose for the moment that (\ref{eqn:nonsing}) holds for some $j,k$.  By Lemma \ref{lem:118}, $E$ intersects  $H_0$, $H_j$, and $H_k$ tangentially.  The intersection with $H_0$ is a point $x=(x_1,\ldots,x_d)$ that satisfies $x_j\geq 0$ and $x_1+\cdots+x_d=1$, so that 
$$
x_j+x_k = 1-\sum_{l\neq j,k}x_l =: \lambda^2 \in(0,1].
$$
The intersection with $H_j$  (resp. $H_k$) is a point $x$ such that $x_j=0$  (resp. $x_k=0$).   Write $y=\pi(x)$ where $y=(y_1,y_2)$; then the projection $\pi(E)$ is an ellipse in $\RR^2$ that intersects the boundary of $\lambda^2\Sigma_2$ in each of the  lines $y_1=0$, $y_2=0$, and  $y_1+y_2=\lambda^2$.   Thus $\pi(E)$ is inscribed in $\lambda^2\Sigma_2$, so by Proposition \ref{prop:117} it must be a Newton ellipse for $\lambda^2\Sigma$. 
Hence by Lemma \ref{lem:114},  $a_j=2|c_j|^2$ and $a_k=2|c_k|^2$.

If $c_l\neq 0$ for any other $l\neq j,k$ then at least one of the projections $z\mapsto(z_j,z_l)$ or $z\mapsto(z_k,z_l)$ is nonsingular on $P_E$, according to whether $c_l$ is not a real multiple of $c_j$ or $c_k$.  Then, running the above argument using this projection yields $a_l=2|c_l|^2$.

If $c_l=0$, we use the fact that $E$ intersects $H_l$, which occurs at a point whose $l$-th component is zero.  That is, there is some value of the parameter $\zeta=e^{i\theta}$, $\theta\in\RR$, for which
$$
0 = a_l + c_le^{i\theta} + \bar c_le^{-i\theta}.
$$
Hence if $c_l=0$ then $a_l=0$ too.  

Altogether, $a_j=2|c_j|^2$ for all $j=1,\ldots,d$, so by Lemma \ref{lem:114}, $E$ is a Newton ellipse.  

Finally, in case (\ref{eqn:nonsing}) fails for all pairs of components of $c$,  we perturb the eccentricity and orientation of $E$ (i.e. the parameter $c$) to get a sequence of inscribed ellipses, whose $c$ parameters have components satisfying (\ref{eqn:nonsing}), and  converge to the $c$ parameter for $E$.  By the previous part of the proof, each ellipse in the sequence is a Newton ellipse, so the conditions $a_j=2|c_j|$ of Lemma \ref{lem:114} relating the parameters $a$ and $c$ are satisfied.  Let $\tilde E$ be a limiting ellipse of the sequence.    By continuity, the parameters of $\tilde E$ must also satisfy the conditions $a_j=2|c_j|$, and therefore $\tilde E$ is a Newton ellipse by the same lemma.

Also, $\tilde E$ is inscribed.  If not, one could translate it to the interior of $\Sigma$.  But then some non-degenerate ellipse in the sequence whose points are sufficiently close to $\tilde E$ could also be translated to the interior, a contradiction.   Hence by Lemma \ref{lem:uniqueellipse}, $E=\tilde E$, so $E$ itself is a Newton ellipse.  
\end{proof}

\begin{proposition}\label{prop:degenerate}
Proposition \ref{prop:97} also hold for degenerate ellipses.
\end{proposition}

\begin{proof}
Suppose $E$ is a degenerate inscribed ellipse.  Consider a sequence of non-degenerate inscribed ellipses whose major axes are in the same direction as the degenerate ellipse and whose minor axes have length decreasing to $0$.  By Proposition \ref{prop:97}, each of the nondegenerate ellipses is a Newton ellipse.  A limiting ellipse $\tilde E$ of the sequence is, by construction, a degenerate ellipse with the same $c$ parameter as $E$.  By the same arguments given at the end of the previous proof,  $\tilde E$ is a Newton ellipse for $\Sigma$ as well as an inscribed ellipse for $\Sigma$.  By Lemma \ref{lem:uniqueellipse}, $E=\tilde E$.

If $E$ is a degenerate Newton ellipse, one can show that $E$ is inscribed by running the same approximation argument as above, using the reverse implication in Proposition \ref{prop:97}.
\end{proof}
Thus Newton ellipses are precisely the extremal ellipses for the standard simplex. 

\begin{proposition}\label{prop:98}
There is a continuous foliation of $\CC^d\setminus\Sigma$ by complexifications of Newton ellipses,
$$
f(\zeta) = a + c\zeta + \bar c/\zeta 
$$
with $a_j=2|c_j|$ for $j=1,\ldots,d$, such that $V_{\Sigma}(f(\zeta))=\bigl| \log|\zeta| \bigr|$. 
\end{proposition}

\begin{proof}
Let $z\in\CC^d\setminus \Sigma$.  Then $z=Q(w)$ for some $w\in\CC^d\setminus B_{\RR}$, and $w\in E_{\CC}$, the complexification of a Hooke ellipse $E$.  Then $z\in Q(E_{\CC})$, which is the complexification of a Newton ellipse.  Thus the collection of (complexified) Newton ellipses covers $\CC^d\setminus\Sigma$.  

Let $E_{1,\CC},E_{2,\CC}$ be distinct complexified Newton ellipses with $a\in E_{1,\CC}\cap E_{2,\CC}$.  Then $a=b^2$ for some $b\in \tilde E_{1,\CC}\cap \tilde E_{2,\CC}$, where $\tilde E_{j,\CC}=Q(E_{j,\CC})$.   By Theorem \ref{prop:19}, $b\in B_{\RR}$ so $a\in\Sigma$.  Hence Newton ellipses are disjoint in $\CC^d\setminus \Sigma$. 

Thus complexified Newton ellipses give a foliation of $\CC^d\setminus\Sigma$.  The foliation is continuous because it is the image of a continuous foliation  (of Hooke ellipses for $B_{\RR}$) under the polynomial map $Q$.  

The formula $V_{\Sigma}(f(\zeta))=\bigl|\log|\zeta|\bigr|$ follows easily from (\ref{eqn:VQ}).
\end{proof}

Mapping forward by an invertible linear map, we obtain a foliation of $\CC^d\setminus S$ by complexified ellipses  for any $d$-dimensional simplex $S$ in $\RR^d$.  
\begin{corollary}\label{cor:89}
Let $S\subset\RR^d$ be a simplex.  Let $\calE_S$ be the collection of inscribed ellipses  in $S$.  Then the collection of  complexified ellipses $$\calE_{S,\CC}:= \{E_{\CC}\colon E\in\calE_S\}$$ give a continuous foliation of $\CC^d\setminus S$ on which $V_S$ is harmonic.  Precisely, each ellipse has a parametrization $f\colon\CC^*\to\CC^d$  of the form 
\begin{equation}\label{eqn:cor122}
f(\zeta) = a + c\zeta + \bar c/\zeta, \quad a\in\RR^d, c\in\CC^d\setminus\{0\},
\end{equation}
and  \begin{equation}\label{eqn:log} V_S(f(\zeta)) = \bigl|\log|\zeta|\bigr|. \end{equation}

Conversely, suppose that $f\colon\CC^*\to\CC^d$ is as in (\ref{eqn:cor122}), and  that the real ellipse $E:=\{f(e^{i\theta})\colon\theta\in\RR\}$ is contained in $S$.  If (\ref{eqn:log}) holds when $|\zeta|\neq 1$ then $E$ is inscribed in $S$, otherwise, $V_S(f(\zeta))<|\log|\zeta||$ for each $\zeta$ such that $|\zeta|\neq 1$.
\end{corollary}

\begin{proof}
Suppose one of the vertices of $S$ is at the origin.  Then there is an invertible linear map $L:\RR^d\to\RR^d$ such that 
$L(\Sigma)=S$.  The result holds for $\Sigma$ by the previous proposition.  If $E$ is an inscribed (i.e. Newton) ellipse for $\Sigma$, then $L(E)$ is an inscribed ellipse for $S$.  Extending $L$ as an invertible linear map to $\CC^d$, we see that for $E,\tilde E\in\calE_{\Sigma}$, we have $E=\tilde E$ if and only if $L(E)=L(\tilde E)$.  If $f$ denotes the parametrization of $E_{\CC}$, which is of the form (\ref{eqn:cor122}) and $V_{\Sigma}(f(\zeta))=\bigl|\log|\zeta|\bigr|$, then the parametrization of  $L(E)_{\CC}$ is given by $L\circ f$, which has the same form: $L(a) + L(c)\zeta + \overline{L(c)}/\zeta$, and $V_S(L(f(\zeta)))=V_{\Sigma}(f(\zeta)) =   \bigl|\log|\zeta|\bigr|$.  

It is also easy to see that if $E$ is inscribed in $S$ then $L^{-1}(E)=:\tilde E$ is a Newton ellipse for $\Sigma$.  Hence $\calE_{S}=\{L(\tilde E)\colon \tilde E\in\calE_{\Sigma}\}$, and these ellipses have the desired properties by the previous paragraph.  

If none of the vertices are at the origin, we translate a vertex $b$  of $S$ to the origin and then apply a linear map $L$ as in the first part of the proof.  The corresponding ellipses have the parametrization $\zeta\mapsto (L(a)-b) +   L(c)\zeta + \overline{L(c)}/\zeta.$ 

We prove the last statement by contraposition.  Suppose $E\subset S$ is not inscribed in $S$.  By Lemma \ref{lem:118} there is a face of $S$ that does not intersect $E$.  We can shrink $S$ to a smaller simplex $\tilde S\subset S$ by translating supporting hyperplanes of $S$ in interior normal directions, so that $E\subset\tilde S$ and $E$ meets all but one face of $\tilde S$.  By making a linear change of coordinates as in the previous paragraph, we may reduce to the case in which $\tilde S=\Sigma$  and $E$ does not intersect the face contained in the hyperplane $\{(x_1,\ldots,x_d)\in\RR^d\colon x_1+\cdots+x_d=1\}$.  Then, $E$ is inscribed in $t^2\Sigma$ for some $t\in(0,1)$, so is a Newton ellipse.  The preimage of $E$ under the square map is the union of  two Hooke ellipses for $tB_{\RR}$ with parametrizations $g_{\pm}(\zeta)= b_{\pm}\zeta+\bar{b_{\pm}}/\zeta$, where $b_{\pm}$ are the square roots of $c$ as in (\ref{eqn:cor122}).  Put $\zeta=\eta^2$.  If $|\zeta|>1$ then 
$$
\log|\zeta|=2\log|\eta|= 2V_{tB_{\RR}}(g_{\pm}(\eta)) = V_{t^2\Sigma}(f(\zeta)) > V_{\Sigma}(f(\zeta))=V_{\tilde S}(f(\zeta))\geq V_S(f(\zeta)), 
$$  
i.e., $V_S(f(\zeta))<\log|\zeta|$, and (\ref{eqn:log}) fails.  A similar calculation works when $|\zeta|<1$ too.
\end{proof}

As with $\Sigma$, we will call ellipses inscribed in a simplex $S$ \emph{Newton ellipses} for $S$.  It is easy to see that Lemma \ref{lem:118} and Propositions \ref{prop:97}--\ref{prop:98} also hold for $S$ (map to $\Sigma$, as in the proof above).    The center of a Newton ellipse for $S$ (given by $a$ in (\ref{eqn:cor122}))  is uniquely determined by the parameter $c$, as it is for $\Sigma$.   We recover the geometric fact that an inscribed ellipse in a simplex for a given eccentricity and orientation is unique.

\section{Ellipses inscribed in a convex polytope} \label{sec:inscribed}


Let $K\subset\RR^d$ be a compact convex $d$-dimensional polytope.   We will relate ellipses inscribed in $K$ to supporting simplices and strips.  An ellipse in $\RR^d$ may be characterized as the image $E=L(C)$ of the unit circle $$C=\{(\cos\theta,\sin\theta)\colon\theta\in\RR\}\subset\RR^2$$  under an  affine map 
$$
 C\ni (\cos\theta,\sin\theta) \stackrel{L}{\longmapsto} \ba_1 + \ba_2\cos(\theta) + \ba_3\sin(\theta),
$$
where $\ba_j\in\RR^d$, $j=1,2,3$.   In complex notation, identifying $(\cos\theta,\sin\theta)\in\RR^2$ with $e^{i\theta}\in\CC$, we may write this as 
\begin{equation} \label{eqn:ellipse}
e^{i\theta}\mapsto a +  c e^{i\theta} + \bar ce^{-i\theta}
\end{equation}
where $a=\ba_1$ and $c=\tfrac{1}{2}(\ba_2-i\ba_3)\in\CC^d$.  (Compare equation (\ref{eqn:cor122}).)  



\medskip

Clearly, a non-degenerate inscribed ellipse $E$, considered as an element of the collection of all ellipses contained in $K$ with the same eccentricity and orientation as $E$,  is one that bounds the largest possible region.   Otherwise one could translate $E$ to $\intt(K)$ and expand it slightly to get a larger ellipse still contained in $K$, with the same eccentricity and orientation. Similarly, a degenerate ellipse is the line segment contained in $K$ of greatest length for a given direction.

The next lemma shows, by a simple compactness argument, that there is an ellipse inscribed in $K$ of any eccentricity and orientation.  
\begin{lemma}\label{lem:ellipseexist}
Let $c\in\CC^d$.  Then there exists $r>0$ such that for $c':=rc$,
\begin{equation}\label{eqn:lem101}
e^{i\theta}\mapsto a + c'e^{i\theta}+\bar{c'}e^{-i\theta}
\end{equation}
parametrizes an ellipse inscribed in $K$, for some $a\in K$.
\end{lemma}

\begin{proof}
Define the function $r_c\colon K\to[0,\infty)$ by letting $r_c(a)$ be the largest value of $r\geq 0$ such that the set parametrized by (\ref{eqn:lem101}) for $c'=rc$ is contained in $K$.  

Given $a\in K$ let $a_n:=a+v_n$ ($|v_n|\to 0$) be a sequence of vectors such that $r_n:=r(a_n)$ decreases  to $$\limsup_{|v|\to 0}r_c(a+v)=:\tilde r.$$  Then each point of the ellipse  $E:= \{ a+\tilde rce^{i\theta} + \tilde r\bar ce^{-i\theta}\colon \theta\in\RR\}$ is a limit point of the union $\bigcup_n E_n$, where $E_n=\{ a_n + r_nce^{i\theta}+r_n\bar{c}e^{-i\theta}\colon \theta\in\RR\}.$  We have $E\subset K$ by compactness, so $\tilde r\leq r_c(a)$.  

Hence $r_c$ is upper semicontinuous, so it attains a maximum on $K$, say at a point $a_c$.  The ellipse parametrized by $e^{i\theta}\mapsto a_c + rce^{i\theta}+r\bar{c}e^{-i\theta}$ (where $r=r_c(a_c)$)  is our desired inscribed ellipse.  If not, we could translate it to an ellipse contained in $\intt(K)$; then $r_c(a)>r_c(a_c)$, where $a$ is the center of the translated ellipse, and obtain a contradiction.
\end{proof}

\begin{remark}\rm
The above argument requires only that one can translate points and remain inside $K$.  Hence the lemma holds for any convex body.  But we must restrict to polytopes in order to apply the results of Sections \ref{sec:simplices} and \ref{sec:strips}.
\end{remark}

We want to relate inscribed ellipses to supporting simplices and strips.
Let $E$ be an inscribed ellipse in $K$. Then $E$ intersects the boundary of $K$ in finitely many points.  
To see this, note that the intersection is nonempty, otherwise $E$ would be in $\intt(K)$.  In fact, $E$ intersects each face in at most one point; the intersection must be tangential in order that $E$ remain in $K$.  Denote by $F_0,\ldots,F_N$ the faces of $K$ that intersect $E$, and for each $j=0,\ldots,N$,
$$
\{a_j\} :=E\cap F_j;
$$ 
then $E\cap\partial K = \{a_0,\ldots,a_N\}$.  Note that some of the $a_j$s may coincide.

For each $j$, let $H_j$ denote the hyperplane containing $F_j$, let $n_j$ denote the unit normal, and let $l_j$ denote the corresponding affine map, as defined previously.




\begin{proposition}\label{prop:ellipsesimplex}
Let $E\subset K$ be an inscribed ellipse.  Then there is a supporting simplex or strip $S\supset K$ such that $E$ is inscribed in $S$.
\end{proposition}

\begin{proof} 
As before, let $F_0,\ldots,F_N$ denote the faces of $K$ that intersect $E$.  Let $H_k,n_k,l_k$ be the corresponding hyperplanes, normals, and affine maps, for each $k=0,\ldots, N$.

We first consider when $\dim(\span\{n_0,\ldots,n_N\})=d$.  Without loss of generality, assume $n_1,\ldots,n_d$ are linearly independent, and let $p_0$ be the point of intersection of     $H_1,\ldots,H_d$.   Let $a$ be as in (\ref{eqn:ellipse}); then $l_k(a)>0=l_k(p_0)$ for all $k=1,\ldots, d$,  and   
\begin{equation}\label{eqn:prop121a}
n_k\cdot(a-p_0)>0   \ \hbox{ for all } k=1,\ldots,d.
\end{equation}
Since $E$ is an inscribed ellipse, a translation in the direction of $a-p_0$ must send at least one point of $E\cap\partial K$ outside of $K$. Say, $a_0$ on the face $F_0$ translates outside $K$; then $l_0( a_0+(a-p_0))<0$, from which
\begin{equation} \label{eqn:prop121b}
n_0\cdot(a-p_0)<0.
\end{equation}
We can find $j\leq d$ linearly independent normal vectors as in Lemma \ref{lem:4.4}.  Say (after relabelling) that these are the vectors $n_1,\ldots,n_j$, and are such that $\{n_k\}_{k=0}^j$ is linearly dependent but any subcollection of $j$ vectors from $\{n_k\}_{k=0}^j$ is linearly independent.  If $j<d$ then by (\ref{eqn:prop121a}), (\ref{eqn:prop121b}), and Proposition \ref{prop:114}, $$S:=\{x\in\RR^d \colon l_k(x)>0 \hbox{ for all } k=1,\ldots,j\}$$ is a supporting strip to $K$ with $j$-dimensional cross-section.  For any vector $b$, $b=b_1+b_2$, where $b_1$ is along the strip and $b_2$ is along the cross-section.  Translating by $b_1$  sends points of $(E\cap\partial K)\subset\partial S$ into points of $\partial S$, while translating by $b_2$ sends points of $E\cap\partial K$ outside $S$.  Hence $E$ is inscribed in $S$.  If $j=d$ then by a similar argument, $S$ is a supporting simplex such that $E$ is inscribed in $S$.

Now suppose $\dim(\span\{n_0,\ldots,n_N\})<d$.  Then $V:=\span\{n_0,\ldots,n_N\}$ is a space of dimension $d_V\leq N$.  Let $K_V=\pi_V(K)$, $E_V=\pi_V(E)$ where $\pi_V\colon\RR^d\to V$ denotes the orthogonal projection.  Apply the previous argument to $K_V,E_V$ in  $V\simeq\RR^{d_V}$, and denote by $S_V$ the strip or simplex obtained (with say, $j$-dimensional cross-section, $j\leq d_V$).  Then, it is straightforward to verify that
$$
S:=S_V + V^{\perp},
$$
where $V^{\perp}$ denotes the orthogonal complement to $V$, is a supporting strip to $K$ in $\RR^d$ with $j$-dimensional cross-section, and $E$ is inscribed in $S$.
\end{proof}

For the rest of this section, let $K\subset\RR^d$  be a compact convex polytope such that the normals to the faces of $K$ satisfy the linear independence condition in Theorem \ref{thm:simplices}.

\begin{lemma}\label{lem:ellipseunique}
For a fixed $c\in\CC^d$, the ellipse given by Lemma \ref{lem:ellipseexist} is unique.  Consequently, the center $a=a(c)$ is a well-defined function of $c$.  
\end{lemma}

\begin{proof}
If there were two such ellipses, then by the proof of Lemma \ref{lem:ellipseexist} they must have the same value of $c'=rc$ in (\ref{eqn:lem101}), given by $r=\max_{a\in K} r_c(a)$.  So they are translates of each other; say one ellipse, $E_0$, is centered at $a$ and the other, $E_1$, at $a+v$.     By the previous proposition, $E_0$ is inscribed in some supporting simplex $S\supset K$, and it is easy to see that $E_1$ must also be inscribed in this same $S$.  We have two Newton ellipses for $S$ with the same eccentricity and orientation, contradicting the fact that such ellipses are unique.
\end{proof}

Recall that $\calS(K)$ denotes the collection of supporting simplices (and strips) of $K$.  Also, write 
\begin{equation}\label{eqn:thm91}  K=\{x\in\RR^d\colon l_j(x)\geq 0 \hbox{ for all }j=1,\ldots,N\},\end{equation}
where the $l_j$ are affine maps such that each hyperplane $H_j=\{l_j=0\}$ contains a face of $K$ of codimension 1.  
We now prove a preliminary version of our main theorem for polytopes with a small number of faces.
\begin{theorem} \label{thm:main}
Let $K\subset\RR^d$ be a compact convex polytope whose faces satisfy the linear independence condition in Theorem \ref{thm:simplices}. Suppose $N\in\{d+1,d+2\}$ in (\ref{eqn:thm91}).  Then we have the formula
\begin{equation}\label{eqn:mainthm}
V_K(z) = \max\{ V_{S}(z)\colon S \in\calS(K) \}.
\end{equation}
\end{theorem}


\begin{proof}  When $N=d+1$, $K$ is itself a simplex and the theorem is trivial.

Suppose $N=d+2$.  Denote the right-hand side of (\ref{eqn:mainthm}) by $W=W(z)$.  Since $V_S\leq V_K$ for all $S\in\calS(K)$, we have $W\leq V_K$. It remains to show that $V_K\leq W$.

We have $K=S_1\cap S_2$ where $S_1,S_2$ are the simplices given by 
$$S_1=\{ x\in\RR^d\colon l_j(x)\geq 0 \hbox{ for all }j\neq d+1\},\  S_2=\{ x\in\RR^d\colon l_j(x)\geq 0 \hbox{ for all }j\neq d+2\}$$
(after a possible relabelling of the $l_j$s).  By a linear change of coordinates, we may assume that $l_j(x)=x_j$ for all $j=1,\ldots,d$.  Then in these coordinates, 
$$l_{d+1}(x)=a_0-a_1x_1-\cdots-a_dx_d, \quad  l_{d+2}(x)=b_0-b_1x_1-\cdots-b_dx_d,$$
where $a_j,b_j>0$ for all $j$.  Positivity of $a_j,b_j$ follows from the condition on the normals to the hyperplanes $H_{d+1}$ and $H_{d+2}$ as given in (\ref{eqn:lem19}) of Lemma \ref{lem:119}.  Then, the inverse images of $S_1$ and $S_2$ under the square map are solid ellipsoids about the origin, $A_1=Q^{-1}(S_1)$ , $A_2=Q^{-1}(S_2)$, with 
$$\begin{aligned}
\partial A_1 &=\{x\in\RR^d\colon a_1x_1^2+\cdots+ a_dx_d^2=a_0\},\\
 \partial A_2 &=\{x\in\RR^d\colon b_1x_1^2+\cdots+b_dx_d^2=b_0\}.
\end{aligned}$$  Let $A=A_1\cap A_2$; then $A=Q^{-1}(K)$.   By Proposition \ref{thm:ellipsoid}, $V_{A}=\max\{V_{A_1},V_{A_2}\}$.  Applying Proposition \ref{prop:61}(2) to both sides, $V_K=\max\{V_{S_1},V_{S_2}\}\leq W,$ which proves the theorem when $N=d+2$.
\end{proof}

Following \cite{burnslevmau:exterior}, we relate the geometry of an ellipse $E$ to the growth of an extremal function on its complexification.  If $E$ is given by (\ref{eqn:ellipse}),  $E_{\CC}$ is given by (\ref{eqn:cor122}).  

 It is convenient to use projective space.  Let $\CC^d\subset\PP^d$ via the standard inclusion $z\hookrightarrow[1:z]$ in homogeneous coordinates, and let $H_{\infty}:=\PP^d\setminus\CC^d$.  The parametrization of $E_{\CC}$ may be rewritten in homogeneous coordinates as 
$$
f(\zeta) = [\tfrac{1}{\zeta}:\tfrac{a}{\zeta} +c + \tfrac{\bar c}{\zeta^2}], 
$$
which extends (taking the limit as $|\zeta|\to\infty$) across $H_{\infty}$ as $f(\infty) = [0:c]$.  Let us write the closure of $E_{\CC}$ in $\PP^d$ also as $E_{\CC}$. 

We study the growth of the extremal function by extending a modification of it across $H_{\infty}$.  Given a compact set $K\subset\CC^d\subset\PP^d$, define
$$
W_K([1:z]):=V_K(z)-\log|z|.
$$
Since $V_K\in\calL$, $W_K$ is uniformly bounded above on $\CC^d$, so it extends across $H_{\infty}$ in $\PP^d$:
$$
W_K([0:c]) = \limsup_{t\to 0,z\to c} (V_K(z/t) -\log|z/t|) .
$$
When $K$ is regular, it follows from Corollary 4.4 of \cite{bloomlevmau:robin} that the limsup is in fact a limit and $W_K$ is continuous across $H_{\infty}$.  Hence 
\begin{equation}\label{eqn:104}
W_K([0:c]) = \rho_K(c)-\log|c|.
\end{equation}

When $K$ is a simplex, the value of $W_K$ across $H_{\infty}$ is related to the parametrization of a Newton ellipse.

\begin{lemma}\label{prop:92}
Suppose the ellipse $E$ is inscribed in a simplex $S$ and its complexification $E_{\CC}$ is parametrized by  
\begin{equation}\label{eqn:lem10.2}
f(\zeta) = a + c\zeta + \bar c/\zeta.
\end{equation}
Then $W_S([0:c])=-\log|c|$ and $\rho_S(c)=0$.
\end{lemma}

\begin{proof}
Since $E$ is inscribed in $S$, we have $V_S(f(\zeta))=\log|\zeta|$ by Corollary \ref{cor:89}, so  
$$\begin{aligned}
W_{S}([0:c]) & = \lim_{t\to 0,z\to c}\left(V_{S}(z/t)-\log|z/t| \right) \\
&=\lim_{|\zeta|\to\infty}\left( V_{S}(c\zeta) - \log|c\zeta|\right) \\
&= \lim_{|\zeta|\to\infty}\left( V_{S}(f(\zeta)+O(1))-\log|c\zeta| \right)\\
&=\lim_{|\zeta|\to\infty}\left( V_{S}(f(\zeta)) + O(\tfrac{1}{|\zeta|}) -\log|\zeta|\right) -\log|c| 
= -\log|c|.
\end{aligned}$$
By (\ref{eqn:104}), $\rho_S(c)=0$.  
\end{proof}

\def\calC{\mathcal{C}}

\begin{proposition} \label{prop:106}  
Let $K\subset\RR^d$ be a compact convex polytope whose faces satisfy the linear independence condition in Theorem \ref{thm:simplices}. 

Let $E$ be an ellipse inscribed in $K$.  If $E_{\CC}$ is parametrized as in equation (\ref{eqn:lem10.2}), then  $V_K(f(\zeta))=\log|\zeta|$ for all  $|\zeta|\geq 1$, $\rho_K(c)=0$, and $W_K([0:c])=-\log|c|$. 

If $V_K$ satisfies (\ref{eqn:mainthm}) then we also have a converse.   That is, if $E_{\CC}$ is parametrized as in equation (\ref{eqn:lem10.2}), with $E\subset K$, and  $V_K(f(\zeta))=\log|\zeta|$ for all $|\zeta|>1$, then $E$ is inscribed in $K$.  
\end{proposition}

\begin{proof}
If $E$ is inscribed in $K$ then by Proposition \ref{prop:ellipsesimplex} and Theorem \ref{thm:simplices}, it is inscribed in some supporting simplex. By Corollary \ref{cor:89},  $V_S$ is harmonic on $E_{\CC}\setminus E$ with $V_S(f(\zeta))=\log|\zeta|$.   The function $V_K-V_S$ is  bounded, nonnegative and subharmonic  on $E_{\CC}\setminus E$, and goes to zero on $E$.  Hence by the maximum principle, $V_K-V_S$ is identically zero on all of $E_{\CC}$.  So on $E_{\CC}$,  $V_K=V_S$.
The same calculation as Lemma \ref{prop:92} then yields $\rho_K(c)=0$ and $W_K([0:c])=-\log|c|$.

We prove the last statement by contraposition.  
 Let $f$ parametrize $E_{\CC}$ as in (\ref{eqn:lem10.2}), $E\subset K$, and suppose $E$ is not inscribed in $K$.  Then $E$ is not inscribed in $S$ for any supporting simplex $S$, and $V_S(f(\zeta))<\log|\zeta|$ by the last part of Corollary \ref{cor:89}.  Hence by equation (\ref{eqn:mainthm}), $V_K(f(\zeta))<\log|\zeta|$.
\end{proof}


\section{The Robin exponential map} \label{sec:robinexpmap}

In this section, we again restrict to the case in which our convex polytope $K\subset\RR^d$ satisfies the linear independence condition in Theorem \ref{thm:simplices}, so that $\calS(K)$ is a collection of $d$-dimensional supporting simplices.

\begin{proposition}\label{prop:107}
Let $E_{\CC},\tilde E_{\CC}$ be complexifications of ellipses $E,\tilde E$ inscribed in $K$.  Then $(E_{\CC}\cap\tilde E_{\CC})\setminus K=\emptyset$ or $E=\tilde E$.  
\end{proposition}

\begin{proof}
 Let $S\supset K$ be a supporting simplex such that $E$ is inscribed in $S$.  If $\tilde E\neq E$ and $\tilde E$ also happens to be  inscribed in $S$, then the complexifications $E_{\CC}, \tilde E_{\CC}$ are disjoint outside of $S$  by Corollary \ref{cor:89}, and hence disjoint outside of $K$.  Therefore it is sufficient to show the following:
\begin{itemize}
\item[($\dag$)] {\sl If $E,\tilde E$ are inscribed in $K$ and $(E_{\CC}\cap \tilde E_{\CC})\setminus K\neq\emptyset$, then there is a supporting simplex $S\supset K$ such that  $E,\tilde E$ are both inscribed in $S$.} \end{itemize}


To prove ($\dag$), let $S\supset K$ be a supporting simplex such that $E$ is inscribed in $S$.     Let $f,\tilde f$ be the parametrizations of $E_{\CC},\tilde E_{\CC}$ as in (\ref{eqn:lem10.2}).   
Let $w\in\CC^d\setminus S$ and write $w=f(s)=\tilde f(\tilde s)$.  Then $V_K(w)=\log|s|=\log|\tilde s|$, so $|s|=|\tilde s|>1$.  By Proposition \ref{prop:106}, $V_S(f(s))=\log|s|$ since $E$ is inscribed in $S$, which yields 
$$
0 = V_S(w)-\log|s| =   V_S(w)-\log|\tilde s|  
 = V_S(\tilde f(\tilde s))-\log|\tilde s|  =: u(\tilde s).
$$
Since $\tilde E\subset K\subset S$, we have $u(\zeta)= 0$ when $|\zeta|=1$.  By the maximum principle, $u\equiv 0$ on the exterior of the unit disk, so $V_S(\tilde f(\zeta))=\log|\zeta|$ for all $|\zeta|>1$.  By the second part of Corollary \ref{cor:89} (or the second part of Proposition \ref{prop:106} applied to $S$), $\tilde E$ is inscribed in $S$.  This proves ($\dag$), and hence the proposition.
\end{proof}

 Recall that $K_{\rho}=\{z\in\CC^d\colon \rho_K(z)\leq 0\}$.  We want to define the \emph{Robin exponential map} on $\CC^d\setminus K_{\rho}$ by 
\begin{equation}\label{eqn:rexp}
\calR_K(c\zeta):= a + rc\zeta+ r\bar c/\zeta, \end{equation}
where $\rho_K(c)=0$, $|\zeta|>1$, and $a=a(c)$, $r=r(c)$ are chosen as in Lemma \ref{lem:ellipseexist} so that the ellipse $E:=\{a+rce^{i\theta}+ r\bar ce^{-i\theta}\colon \theta\in\RR\}$ is inscribed in $K$.  
We can do this under the condition that formula (\ref{eqn:mainthm}) holds, in order to apply the second part of Proposition \ref{prop:106}.

\begin{lemma}\label{lem:10.9}
Suppose (\ref{eqn:mainthm}) holds for $V_K$.  Then the Robin exponential map is well-defined,  $\calR_K\colon\CC^d\setminus K_{\rho} \to \CC^d\setminus K$,  and for $c\in\partial K_{\rho}$, $|\zeta|>1$,
\begin{equation*}
\calR_K(c\zeta)= a(c) + c\zeta+ \bar c/\zeta.
\end{equation*}
\end{lemma} 

\begin{proof}
Let $z\in\CC^d\setminus K_{\rho}$.  Then $z=c\zeta$ for some $c\in\partial K_{\rho}$ and $|\zeta|>1$.  To show that this is well-defined, suppose $z=c'\zeta'$. Then $c'=(\zeta/\zeta')c$ so $[0:c]=[0:c']$.   By Proposition \ref{prop:106}, $\log|c|=\log|c'|\  (=-W_K([0:c])$.  Hence $c'=e^{i\theta}c$  and $\zeta'=e^{-i\theta}\zeta$ for some $\theta\in\RR$, so 
$c/\zeta=c'/\zeta$.  Proposition \ref{prop:106} also shows that $r=1$.  If $a(c)\neq a(c')$ these would be the centers of two ellipses inscribed in $K$ with the same eccentricity and orientation, contradicting Lemma \ref{lem:ellipseunique}.  So $a(c)=a(c')$ and $\calR_K(c\zeta)=\calR_K(c'\zeta')$, showing that we have a well-defined function.  Finally,  $\calR_K(c\zeta)\not\in K$ since $|\zeta|>1$.  So $\CC^d\setminus K$ is a valid codomain for $\calR_K$. 
\end{proof}

\begin{theorem}\label{thm:Robinexpmap}
Under the hypotheses of Lemma \ref{lem:10.9}, the Robin exponential map is a homeomorphism.
\end{theorem}

\begin{proof}
We first show $\calR_K$ is one-to-one.  Suppose $z\neq z'$.  Write $z=c\zeta$ and $z'=c'\zeta'$ where $c,c'\in\partial K_{\rho}$ and $|\zeta|,|\zeta'|>1$.  If $c\neq e^{i\theta}c'$ for any $\theta\in\RR$ then $\calR_K(c\zeta)$ and $\calR_K(c'\zeta')$ lie on the complexifications of different inscribed ellipses $E_c,E_{c'}$.  By disjointness outside $K$ (Proposition \ref{prop:107}), $\calR_K(c\zeta)\neq\calR_K(c'\zeta')$.  

If $c=c'e^{i\theta}$ for some $\theta$, then $z'=c\eta$ where $\eta=e^{-i\theta}\zeta'$; and $z\neq z'$ implies $\zeta\neq\eta$. Then $\calR_K(c\zeta)\neq\calR_K(c\eta)$ because they are points on the ellipse $E_c$ for different parameters.  

In either case $\calR_K(z)\neq\calR_K(z')$, so $\calR_K$ is one-to-one.

Continuity of $\calR_K$ will follow from the continuity of $(c,\zeta)\mapsto a(c)+c\zeta +\bar c/\zeta$.  When $K=S$ is a simplex it is easy to see that $a(c)$ is continuous in $c$. (Indeed, we have the explicit formula $a(c)=(2|c_1|,\ldots,2|c_d|)$ for the standard simplex.)  

To show continuity of $a(c)$ when $K$ is a polytope, consider a sequence $c_n\to c$.   We need to show that $a(c_n)\to a(c)$.  Let $S_1,\ldots,S_j$ be all of the supporting simplices of $K$ such that $E_c$ is inscribed in each simplex.  By $(\dag)$ in the proof of Proposition \ref{prop:107}, each ellipse $E_{c_n}$ is inscribed in one of $S_1,\ldots,S_j$.  Partition $\{c_n\}$ into subsequences $\{c_{n,1}\},\ldots,\{c_{n,j}\}$ such that $E_{c_{n,k}}$ is inscribed in $S_k$ for each $k=1,\ldots,j$.   Then $a(c_{n,k})\to a(c)$ for each $k$ as $n\to\infty$.  Combining all the subsequences back into the original sequence, $a(c_n)\to a(c)$.  So $a(c)$ is continuous in $c$.  It follows that $\calR_K$ is continuous.

Since $\calR_K$ is continuous and injective, it is a homeomorphism onto its image by the domain invariance theorem in topology.  The same argument as in the proof of Proposition \ref{prop:jhomeo} shows that $\calR_K$ is onto $\CC^d\setminus K$.
   \end{proof}


\section{The extremal function of a convex polytope}\label{sec:computing}

In this section, we prove the following version of our main theorem.

\begin{theorem} \label{thm:mainN}
Let $K\subset\RR^d$ be a compact convex polytope whose faces satisfy the linear independence condition in Theorem \ref{thm:simplices}. Then 
\begin{equation*}
V_K(z) = \max\{ V_{S}(z)\colon S \in\calS(K) \}.
\end{equation*}
\end{theorem}

As in (\ref{eqn:thm91}), write
$$
K=\{x\in\RR^d\colon l_j(x)\geq 0 \hbox{ for all } j=1,\ldots,N\}.
$$
for affine maps $l_1,\ldots,l_N$.  We will prove the theorem by induction on $N$.  The base cases $N\in\{d+1,d+2\}$ have been proved in Theorem \ref{thm:main}.  The argument will also make use of  Lemma \ref{lem:10.9} and Theorem \ref{thm:Robinexpmap}.  These are  consequences of the formula for $V_K$ in the theorem, and may be applied to polytopes for which the inductive hypothesis is assumed to be true.


\begin{proof}[Proof of Theorem \ref{thm:mainN} by induction]


Let $W(z):=\max\{ V_{S}(z)\colon S \in\calS(K) \}$.  As before, we want to show that $V_K\leq W$.  

Let $z_0\in\CC^d\setminus K$.  Choose $K_0\supset K$ given by $N-1$ of the affine maps $l_j$, for which $V_{K_0}(z_0)=W(z_0)$.  (We are using the inductive hypothesis and the fact that every  $S\in\calS(K)$ is also in $\calS(L)$ for some $L$ given by $N-1$ of the $l_j$s.)   Also, applying Lemma \ref{lem:10.9} and  Theorem \ref{thm:Robinexpmap} to $K_0$ under induction, the Robin exponential map $\calR_{K_0}$ exists and is onto. Hence there is an ellipse $E_0$ inscribed in $K_0$ such that $z_0=f_0(\zeta_0)\in E_{0,\CC}$ where  $|\zeta_0|>1$.   As before, $f_0$ denotes a parametrization of $E_{0,\CC}$ as in  (\ref{eqn:lem10.2}).\footnote{Recall that in terms of the $c$ parameter, $\calR_{K_0}(c\zeta_0)=z_0$.}

By Proposition \ref{prop:ellipsesimplex} there is also a supporting simplex $S$ of $K_0$ such that $E_0$ is inscribed in $S$.  After a possible relabelling of affine maps, we may assume that
$$\begin{aligned}
K_0 &=\{x\in\RR^d\colon l_j(x)\geq 0 \hbox{ for all } j=1,\ldots,N-1\}, \\ S&=\{x\in\RR^d\colon l_j(x)\geq 0\hbox{ for all }j=1,\ldots, d+1\}.
\end{aligned}$$
Now consider the convex polytope 
$$
K_1:=\{x\in\RR^d\colon l_j(x)\geq 0 \hbox{ for all } 1\leq j\leq N,\ j\neq N-1\}.
$$
Then $K=K_0\cap K_1$ by construction, and
$$K_{2}:=K_0\cup K_1=\{x\in\RR^d\colon l_j(x)\geq 0 \hbox{ for all } j=1,\ldots,N-2\}$$ is also a convex polytope.  Moreover, since $N\geq d+3$, $K_2\subset S$, so $E_0$ is also inscribed in $K_2$.  

We claim that $E_0$ is inscribed in $K_1$.  For the purpose of contradiction, suppose not.  By induction, as with $K_0$, there is an ellipse $E_1$  inscribed in $K_1$ whose complexification $E_{1,\CC}$ has parametrization $f_1$ as in (\ref{eqn:lem10.2}), such that $z_0\in E_{1,\CC}$ and  $V_{K_1}(f_1(\zeta))=\log|\zeta|$ when $|\zeta|>1$.  Write $z_0=f_1(\zeta_1)$ where $|\zeta_1|>1$.  Then $E_1\subset K_2$ and $E_1\neq E_0$.  

If $E_1$ is not inscribed in $K_2$, then by induction applied to $K_2$ and Proposition \ref{prop:106},  $V_{K_2}(z_0)<\log|\zeta_1|= V_{K_1}(z_0)$.  On the other hand, by the same proposition, $V_{K_2}(z_0)=\log|\zeta_0|=V_{K_0}(z_0)$ since $E_0$ is inscribed in both $K_0$ and $K_2$.  But then $V_{K_1}(z_0)>V_{K_0}(z_0)=W(z_0)$.  Hence there is a simplex $S_1\in\calS(K_1)$ such that $V_{S_1}(z_0)>V_{S}(z_0)$ for all $S\in\calS(K)$.  But it is easy to see that $\calS(K_1)\subset\calS(K)$, and this gives a contradiction.

So $E_1$ must be inscribed in $K_2$.  But then, $E_0$ and $E_1$ are ellipses inscribed in $K_2$ with the property that $z_0\in E_{0,\CC}\cap E_{1,\CC}$ and  $V_{K_2}(z_0)>0$.  Hence by Proposition \ref{prop:107}, $E_0=E_1$.

So $E_0$ is inscribed in $K_1$, as claimed.  But since $E_0$ was in $K_0$ to begin with, it means that $E_0$ is inscribed in $K$.  Hence $f_0(\partial\Delta)\subset K$.

Now, consider the function $\varphi(\zeta):= W(f_0(\zeta))-\log|\zeta|$ for $|\zeta|\geq 1$.  Then  $\varphi$ is continuous on $\CC\setminus\Delta$, identically zero on $\partial\Delta$ and $\varphi(\zeta_0)=0$.  Hence $\varphi\equiv 0$ on $\CC\setminus\Delta$ by the maximum principle.  Another application of the maximum principle shows that the function $\psi(\zeta):=V_K(f_0(\zeta))-\log|\zeta|$ satisfies $\psi\leq 0$ on $\CC\setminus\Delta$.  In particular, $\psi(f_0(\zeta_0))\leq 0 = \varphi(f_0(\zeta_0))$, so $V_K(z_0)\leq W(z_0)$.  
\end{proof}

\begin{remark}\rm 
The case $N=d+2$ had to be treated separately in Theorem \ref{thm:main}, because otherwise the set $K_2$ in the above proof would involve only $d$ affine maps, which gives an unbounded cone.  Also, since $V_{K_0}(f_0(\zeta))=V_{K_1}(f_0(\zeta))=\log|\zeta|$ (from the fact that $E_0$ is inscribed in both $K_0$ and $K_1$), the proof of the inductive step shows that $V_K=\max\{V_{K_0},V_{K_1}\}$ .  This is not true for a general intersection of convex bodies; see Figure 4.
\end{remark}

\begin{figure}
\begin{center}
\includegraphics[height=5cm]{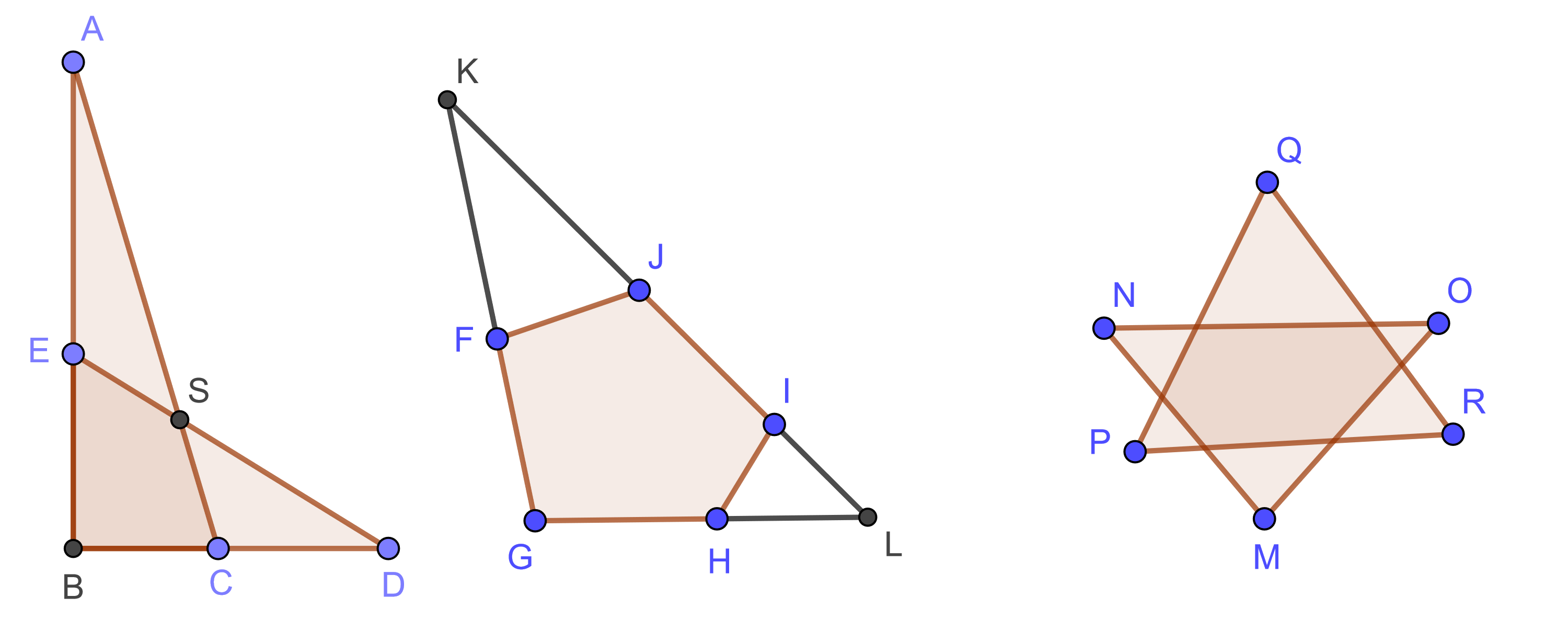}
\end{center}
\caption{Theorem \ref{thm:main} covers quadrilateral BCSE in the left picture (a base case).  In the middle picture, let $K$ be the pentagon FGHIJ,  and let $K_0$, $K_1$ be the quadrilaterals KGHI, FGLJ respectively, so that $K_2$ is the triangle KGL.  Then $V_K=\max\{V_{K_0},V_{K_1}\}$.  In  the right picture, the intersection of the triangles $T_1=\text{MNO}$ and $T_2=\text{PQR}$ is convex but their union is not.  We cannot conclude that  $V_{T_1\cap T_2}=\max\{V_{T_1},V_{T_2}\}$ .  (By inscribing ellipses in $T_1\cap T_2$, it is easy to see that this formula is false).}
\end{figure}

Using Theorems \ref{thm:mainN} and \ref{thm:Vsimplex}, we can compute the extremal function of a compact convex polytope explicitly.  We illustrate this on a simple example in $\RR^2$.    

Let $K$ be the convex hull of the points $\{ (0,0),(1,0),(3/4,3/4),(0,1)\}$, which is a quadrilateral with these points as vertices.  We determine the supporting lines to $K$ containing the edges of the quadrilateral to be 
$$
z_1=0,\ z_2=0,\ 3z_2+z_1=3,\ z_2+3z_1 =3.
$$
Pairs of lines intersect at the vertices as well as the points $(3,0),(0,3)$. 
If we let 
$$\ell_1(z)=z_1, \     \ell_2(z)=z_2,   \      \ell_3(z) = 3-z_1-3z_2,\  \ell_4(z) = 3-3z_1-z_2, $$
then 
$K=\{z\in\RR^2\colon \ell_j\geq 0 \hbox{ for all } j\}$.   This is the same set described in Remark \ref{rem:210}. Let $S_j$, $j=1,\ldots,4$ be as in that remark.  Note that only $S_3$ and $S_4$ are in $\calS(K)$ (i.e., are triangles).

For $z=(z_1,z_2)\in\CC^2$, barycentric coordinates on $S_3$ are found by solving 
$$
\begin{bmatrix}
z_1 & z_1-1 & z_1 \\
z_2 & z_2 & z_2-3\\
1&1&1
\end{bmatrix}   
\begin{bmatrix} 
\lambda_0\\ \lambda_1\\ \lambda_2
\end{bmatrix} \ = \ 
\begin{bmatrix} 
0 \\ 0 \\  1
\end{bmatrix}
$$
to obtain 
$$\lambda_0 = 1-z_1-\tfrac{1}{3}z_2, \    \lambda_1=z_1,  \  \lambda_2=\tfrac{1}{3}z_2.$$
Similarly, barycentric coordinates on $S_4$ are given by
$$
\lambda_0=1-\tfrac{1}{3}z_1-z_2, \   \lambda_1=\tfrac{1}{3}z_1, \  \lambda_2= z_2.  
$$
Hence
\begin{equation*}
\begin{aligned}
 V_K(z) = \max\{ V_{S_3}(z), V_{S_4}(z) \} 
 &   = \max\left\{ \log h(|1-z_1-\tfrac{1}{3}z_2|+|z_1|+|\tfrac{1}{3}z_2|),  \right. \\
& \qquad\qquad \  \left. \log h(|1-\tfrac{1}{3}z_1-z_2|+|\tfrac{1}{3}z_1|+|z_2| )  \right\}.
\end{aligned} \end{equation*}

Such computations can be automated in MATLAB or similar numerical linear algebra software.   In Theorem \ref{thm:maingen} we will extend Theorem \ref{thm:mainN} to a general polytope $K$, where $\calS(K)$ may include strips.  Assuming the general version of the theorem, we present an algorithm to compute $V_K$.

Suppose the linear maps that determine the $(d-1)$-dimensional faces of $K\subset\RR^d$ are given.\footnote{One can use the vertices of $K$ to compute these (see e.g., \cite{matt:analyze}).}  We then compute all of the simplices and strips in $\calS(K)$:
\begin{enumerate}
\item  Let $j:=d$ and $K_j:=K$.
\item  For each collection of $j+1$ supporting hyperplanes, form the $(j+1)\times j$ matrix $N$ of normal vectors to $K_j$.
\item If each $j\times j$ minor of $N$ is of rank $j$, then we possibly have a supporting simplex.  In that case, check condition (\ref{eqn:lem19}) with $x\in K_j$.  
\item Otherwise, there is a $j\times j$ minor not of rank $j$; map $\RR^j$ to $\RR^{j-1}$ by projecting along a direction orthogonal to the normal vectors of the hyperplanes in this minor.  Let $K_{j-1}:=\pi_j(K_j)$ where $\pi_j$ denotes this projection map.   
\item Replace $j$ by $j-1$.  Repeat from step 2 in $\RR^j$ with $\pi_j(K_j)$   and the hyperplanes in $\RR^j$ that are images under $\pi_{j+1}$ of the hyperplanes in the previous step.  
\item After a finite number of such iterations we will eventually obtain a simplex in $\calS(K_j)$,  $K_j\in\RR^j$ for some $j\geq 1$. The set 
$$   S:= \pi_d^{-1}\circ\cdots\circ  \pi_{j+1}^{-1}(S_j)$$ is then the desired  strip in $\calS(K)$.
\item Repeat until all simplices and strips in $\calS(K)$ have been constructed.
\item Use barycentric coordinates to compute the extremal function of each simplex or strip, and hence compute $V_K$.
\end{enumerate}

We illustrate the algorithm on a simple example that generates a strip.  Consider the trapezium $K=\{z\in\RR^2\colon l_j(z)\geq 0,\ j=1,\ldots,4\}$  with linear maps
$$
l_1(z)=z_1,\  l_2(z)=z_2,\ l_3(z)=1-z_2,\ l_4(z)=3-3z_1-z_2.
$$
The normal vectors are $(1,0),(0,1),(0,-1),(-3,-1)$.  
We look at collections of $3$ normal vectors ($j=2$):
\begin{itemize}
\item  $\{(1,0),(0,1),(0,-1)\}$, which gives $N=\begin{bmatrix} 1 & 0\\  0&1\\ 0&-1   \end{bmatrix}$.  The minor given by the last 2 rows is singular, and a null vector is $(1,0)$. Projection along this vector is $$z-(z\cdot(1,0))(1,0)=(0,z_2)$$ which may be identified with the map $\pi(z)=z_2$.    Linear maps for $K_1:=\pi(K)$ may be derived from the linear maps for $K$:
$$\tilde l_1(t):=l_1(0,t)=0 , \ \tilde l_2(t):=l_2(0,t)=t,\ \tilde l_3(t)=l_3(0,t)=1-t.$$
Clearly only $\tilde l_2,\tilde l_3$ are needed, and yield $K_1:=\pi(K)=[0,1]$.  We get the strip  $S=\pi^{-1}[0,1]=\RR\times[0,1]$.  The extremal function is 
$$V_S(z)=V_{[0,1]}(\pi(z))=V_{[0,1]}(z_2)= \log h(|z_2|+|1-z_2|).$$
{\it Remark.}  Observe that the strip condition (\ref{eqn:simpcondition}) is satisfied. Let $x=(\tfrac{1}{2},\tfrac{1}{2})$, $p_0=(0,0)$.  Then $x-p_0=(\tfrac{1}{2},\tfrac{1}{2})$ and $$(0,-1)\cdot(\tfrac{1}{2},\tfrac{1}{2})=-\tfrac{1}{2}<0,\quad (0,1)\cdot(\tfrac{1}{2},\tfrac{1}{2})=\tfrac{1}{2}>0.$$
All other minors of $N$ are nonsingular.  

\item $\{(1,0),(0,1),(-3,-1)$, which gives $N=\begin{bmatrix} 1 & 0\\  0&1\\ -3&-1   \end{bmatrix}$.  All minors of $N$ are nonsingular.  We check condition (\ref{eqn:lem19}) with  $x=(\tfrac{1}{2},\tfrac{1}{2})$ and $p_0=(0,0)$: 
$$
(-3,-1)\cdot(\tfrac{1}{2},\tfrac{1}{2})=-2<0,\  (1,0)\cdot(\tfrac{1}{2},\tfrac{1}{2})=\tfrac{1}{2}>0,\  (0,1)\cdot(\tfrac{1}{2},\tfrac{1}{2})=\tfrac{1}{2}>0.
$$
Hence we have a simplex.  This is the triangle $S_3$ from before, with  extremal function
$$
V_{S_3}(z)=\log h(|1-z_1-\tfrac{1}{3}z_2|+|z_1|+|\tfrac{1}{3}z_2|).
$$

\item $\{(1,0),(0,-1),(-3,-1)\}$.  All minors of the associated matrix $N$ are nonsingular.  We check condition (\ref{eqn:lem19}).  Take $x=(\tfrac{1}{2},\tfrac{1}{2})$ and $p_0=(0,1)$, so that $x-p_0=(\tfrac{1}{2},-\tfrac{1}{2})$.  Then
$$
(-3,-1)\cdot(\tfrac{1}{2},-\tfrac{1}{2})=-2>0,\ (1,0)\cdot(\tfrac{1}{2},-\tfrac{1}{2})=\tfrac{1}{2}>0,\  (0,-1)\cdot(\tfrac{1}{2},-\tfrac{1}{2})=\tfrac{1}{2}>0.
$$
The condition fails, so no simplex is obtained.

\item $\{(0,1),(0,-1),(-3,-1)\}$.  The minor of $N$ given by the first two normal vectors is singular, and all others are nonsingular.  We end up with the same strip as in the first case.
\end{itemize}
Altogether, $\calS(K)$ consists of one simplex and one strip.  The extremal function is 
$$V_K=\max\{ \log h(|z_2|+|1-z_2|), \log h(|1-z_1-\tfrac{1}{3}z_2|+|z_1|+|\tfrac{1}{3}z_2|)\}.$$

\section{The extremal function of a real convex body} \label{sec:blm}

We can now give a new proof of the main theorem in \cite{burnslevmau:pluripotential}.  

\begin{theorem}\label{thm:blm}
Let $K\subset\RR^d$ be a compact convex body.  For each $z\in\CC^d\setminus K$ there exists a complex ellipse  $E_{\CC}\subset\CC^d$  (i.e., an algebraic curve of degree 2, or of degree 1 if degenerate) with parametrization  
\begin{equation}\label{eqn:complexellipse}
\CC^*\ni\zeta \stackrel{f}{\longmapsto} \ba + \bc\zeta+\bar\bc/\zeta\in\CC^d  \quad (\ba\in\RR^d,\ \bc\in\CC^d)
\end{equation}
such that  
\begin{itemize}
\item[(i)]
$z=f(\zeta_z)$ for some $|\zeta_z|>1$,  \item[(ii)] $V_K(f(\zeta))=\log|\zeta|$ for all $|\zeta|>1$, and  
\item[(iii)] $E:=\{f(e^{i\theta})\colon \theta\in\RR\}$ is a real ellipse (or line segment, if degenerate) that is inscribed in $K$. 
 \end{itemize} \end{theorem}

\begin{proof}
When $K$ is a convex polytope as in the previous section we use the Robin exponential map and Theorem \ref{thm:Robinexpmap}.  Precisely, given $z\in\CC^d\setminus K$, let $w=\calR_K^{-1}(z)$; then writing $w=c\zeta_z$ for some $|\zeta_z|>1$ and $c\in\partial K_{\rho}$, we have $z=a(c)+c\zeta_z+\bar c/\zeta_z$, and the map $f(\zeta):=a(c) + c\zeta+\bar c/\zeta$ parametrizes an ellipse that satisfies (i)--(iii).

For a general convex body $K$, we first approximate $K$ from above by a decreasing sequence of polytopes $K_1\supset K_2\supset\cdots$,  $\bigcap K_j=K$, with each $K_j$ as in the previous paragraph.  

Let $z\in\CC^d\setminus K$. Without loss of generality,  $z\in\CC^d\setminus K_j$ for each $j$.  Let 
$$f_j(\zeta) = \ba_j + \bc_j\zeta + \bar\bc_j/\zeta \quad (\ba_j\in\RR^d,\ \bc_j\in\CC^d)$$ 
parametrize a complex ellipse $E_{j,\CC}$ such that  $E_j:=\{f_j(e^{i\theta})\colon \theta\in\RR\}$ is a real ellipse  inscribed in $K_j$, and $z=f_j(\zeta_j)$ for some $|\zeta_j|>1$.   By a normal families argument on the components of $f_j$, we have, passing to a subsequence, that $f_j\to f$ locally uniformly in each component.  In particular, $f$ is given by (\ref{eqn:complexellipse}) where $\ba_j\to\ba$ and $\bc_j\to\bc$.  Local uniform convergence also gives $z=f(\zeta_z)$ where $\zeta_j\to\zeta_z$.  Note that we have uniform convergence on any compact annular region that contains the unit circle.
By the pointwise convergence $V_{K_j}\nearrow V_K$ (Proposition \ref{prop:61}(3)) we have 
$$ V_K(f(\zeta_z)) = V_{K}(z) = \lim_{j\to\infty} V_{K_j}(z) = \lim_{j\to\infty}\log|\zeta_j| = \log|\zeta_z|.$$
Hence (i) and (ii) hold for the ellipse $E_{\CC}$ parametrized by $f$.

It remains to verify (iii), i.e.,  $E:=\{f(e^{i\theta})\colon \theta\in\RR\}$ is inscribed in $K$.   This will follow from the fact that $E_j$ is inscribed $K_j$ for each $j$.  First, a limiting argument gives $E\subset K$, using that every point $w\in E$ is a limit of a sequence $w_j\to w$ with $w_j\in E_j$.  Next, to see that $E$ is inscribed in $K$, suppose not.  Then one can translate $E$ to the interior of $K$ (in $\RR^d$), and this translation sends a small neighborhood of $E$ to the interior of $K$.  Such a neighborhood will contain $E_j$ for sufficiently large $j$, by the local uniform convergence $f_j\to f$. Hence for sufficiently large $j$, the   translation also sends $E_j$ to the interior of $K$.  Since $K\subset K_j$, this contradicts the fact that $E_j$ is inscribed in $K_j$, and (iii) holds.
\end{proof}

Theorem \ref{thm:blm} is the fundamental theorem from which all of the results in \cite{burnslevmau:exterior}, \cite{burnslevmau:extremal}, \cite{burnslevmaurevesz:monge}, regarding the regularity of the extremal function $V_K$ and the computation of  the complex equilibrium measure $(dd^cV_K)^d$, may be derived.

As a corollary, we recover the following result (see \cite{baran:plurisubharmonic}, Theorem 4.2).

\begin{corollary}\label{cor:baran}
Let $K\subset\RR^d$ be a convex body symmetric with respect to the origin.  Then $\CC^d\setminus K$ is foliated by complexifications of Hooke ellipses on which $V_K$ is harmonic.  
\end{corollary}

\begin{proof}
We first suppose that $K$ is strictly convex.

Let $z\in\CC^d\setminus K$.  By Theorem \ref{thm:blm} there is an ellipse $E$ inscribed in $K$ such that $z\in E_{\CC}$, parametrized as in (\ref{eqn:complexellipse}) by a function $f$ (with parameters $\ba,\bc$, say) and $V_{K}(f(\zeta))=\log|\zeta|$ when $|\zeta|>1$.  

Since $K$ is symmetric, the ellipse  $-E$ is also inscribed in $K$, and has the same eccentricity and orientation.  In fact, 
$$
-E = \{-f_1(e^{i\theta})\colon \theta\in\RR\} = \{ -{\bf a} + {\bf c}(-e^{i\theta}) + \bar{\bf c}/(-e^{i\theta})\colon \theta\in\RR\}
$$
which is a translation of the set $E$ by $-2{\bf a}$.   However, because $K$ is strictly convex, any sufficiently small translation of $E$ in the direction of $-2\ba$ must translate it to the interior of $K$.  But this is impossible if $E$ is inscribed.  So $\ba=0$, i.e., $E$ is a Hooke ellipse.   

In case $K$ is not strictly convex, we may approximate it from above by a sequence of strictly convex bodies $K_{j}\searrow K$. Given $z\in\CC^d\setminus K$ we have $z\not\in K_j$ for sufficiently large $j$, and by the first part of the proof there exists a Hooke ellipse $E_j$ with the desired properties.  A similar approximation argument as used in the proof of Theorem \ref{thm:blm} shows that the points of $E_j$ converge to a Hooke ellipse for $E$ with the desired properties.  

We have established that complexifications of Hooke ellipses inscribed in $K$ cover all points of $\CC^d\setminus K$.  If $E,\tilde E$ are Hooke ellipses for $K$ that contain the same point $z\in\CC^d\setminus K$, then by the identity principle (Proposition \ref{prop:19}),  $E_{\CC}=\tilde E_{\CC}$. So the complexifications are disjoint outside $K$.  Hence they give a foliation of $\CC^d\setminus K$.
\end{proof}

   We close the paper by generalizing a couple of earlier results.

\begin{theorem}\label{thm:symmetric}
Let $K_1,\ldots,K_m$ be a finite collection of convex bodies in $\RR^d$ that are symmetric with respect to the origin.  Let $K:=K_1\cap\cdots\cap K_m$.    Then
$$
V_{K} = \max\{V_{K_1},\ldots,V_{K_m}\}.
$$
\end{theorem}

\begin{proof}
Observe that $K=L\cap K_m$, where $L=K_1\cap\cdots\cap K_{m-1}$ is also a convex body that is symmetric with respect to the origin. Then the general formula follows from the case $m=2$ and a straightforward induction. 

Hence suppose $m=2$.   We want to show that $V_K\leq\max\{V_{K_1},V_{K_2}\}=:W$.  The proof is similar to that of Proposition \ref{thm:ellipsoid}.  

 Let $z\in\CC^d\setminus K$.  If $W(z)>V_{K_1}(z)$ or $W(z)>V_{K_2}(z)$ then $W$ is maximal in a neighborhood of $z$.
Otherwise, suppose  $W(z)=V_{K_1}(z)=V_{K_2}(z)$

By Corollary \ref{cor:baran} there are Hooke ellipses $E_1,E_2$ inscribed in $K_1,K_2$ such that $z\in E_{1,\CC}\cap E_{2,\CC}$.  Since $V_{K_1}(z)=V_{K_2}(z)$ we have $E_{1,\CC}=E_{2,\CC}$ by the identity principle. The same maximum principle argument as in the proof of Proposition \ref{thm:ellipsoid} then yields $V_{K_1}=V_{K_2}=W\geq V_K$ at all points of the  complexified ellipse. In particular, $V_K(z)\leq W(z)$. 

So either $W$ is maximal in a neighborhood of $z$ or $V_K(z)\leq W(z)$. By the domination principle, $V_K\leq W$.    
\end{proof}

Finally, we generalize Theorem \ref{thm:mainN}, and obtain the most general form of our main theorem.


\begin{theorem}\label{thm:maingen}
Let $K\subset\RR^d$ be a compact convex polytope.  Then  $$V_K(z)=\max\{ V_S(z)\colon S\in\calS(K)\}.$$
\end{theorem}

\begin{proof}
Let $z\in\CC\setminus K$, and let $E_{\CC}$ be a complex ellipse through $z$ as in Theorem \ref{thm:blm}, parametrized by $f$, with real points $E$ inscribed in  $K$.  By Proposition \ref{prop:ellipsesimplex} there is a simplex or strip $S\in\calS(K)$ such that $E$ is inscribed in $S$. 

 If $S$ is a simplex, then by Theorem \ref{thm:blm}(ii) and Proposition \ref{prop:98}, $$V_K(z)=\log|\zeta_z|=V_S(z)$$  (where $z=f(\zeta_z)$) and the theorem holds.

If $S$ is a strip, then $S=\pi^{-1}(S')$ where $S'$ is a simplex in $\RR^j$ for some $j<d$ and $\pi$ is a projection.  Also, $E_{\pi}:=\pi(E)$ is an ellipse inscribed in $S'$ whose complexification $E_{\pi,\CC}$ has  parametrization $\pi\circ f=:f_{\pi}$, with $f_{\pi}(\zeta_z)  =\pi(z)$.  By Theorem \ref{thm:blm}(ii), Proposition \ref{prop:98} (applied to $S'$), and Lemma \ref{lem:62},
$$
V_K(z)=\log|\zeta_z|=V_{S'}(\pi(z))=V_S(z),
$$
which proves the theorem in this case too.
\end{proof}

\noindent{\bf Acknowledgement.}  I would like to thank the referee for a number of helpful comments that simplified some of the arguments, and also for the formulation of Theorem \ref{thm:symmetric}.



\end{document}